\documentclass[letterpaper,12pt]{article}
\usepackage{natbib} \bibpunct{(}{)}{;}{a}{}{,}
\usepackage{fullpage}
\usepackage{amsmath,amssymb,amsthm}
\usepackage{enumerate}
\usepackage{appendix}
\usepackage{graphicx}
\usepackage{color}
\usepackage{hyperref}
\usepackage{algorithm}
\usepackage{algpseudocode}

\hypersetup{
        colorlinks   = true,
        linkcolor    = blue,
        citecolor    = green
}
\usepackage{authblk}

\newtheorem{theorem}{Theorem}
\newtheorem{lemma}{Lemma}
\newtheorem{remark}{Remark}

\newtheorem{assumption}{Assumption}

\newcommand{\imag}{\mathrm{i}\mkern1mu}

\newcommand{\tti}{{2,\infty}}
\newcommand{\mZero}{{\bf 0}}

\newcommand{\ASE}{\operatorname{ASE}}
\newcommand{\SNR}{\gamma}

\newcommand{\Cov}{\operatorname{Cov}}
\newcommand{\Corr}{\operatorname{Corr}}
\newcommand{\diag}{\operatorname{diag}}
\newcommand{\sgn}{\operatorname{sgn}}

\newcommand{\C}{\mathbb{C}}

\newcommand{\R}{\mathbb{R}}
\newcommand{\bbO}{\mathbb{O}}

\newcommand \Op [1] {O_P \left( #1 \right)}
\newcommand \Omegap [1] {\Omega_P \left( #1 \right)}

\newcommand \op [1] {o_P \left( #1 \right)}

\newcommand{\mLambdaobsd}{\boldsymbol{\Lambda}}
\newcommand{\mLambdaobsdperp}{\mLambdaobsd_{\perp}}
\newcommand{\mSigmaobsd}{\boldsymbol{\Sigma}}
\newcommand{\sigmaobsd}[1]{\sigma_{#1}}
\newcommand{\mTheta}{\boldsymbol{\Theta}}

\newcommand{\calK}{\mathcal{K}}

\newcommand{\Zobsd}{Z}

\newcommand{\Utrue}[1]{U^\star_{#1}}

\newcommand{\ve}{\mathbf{e}}
\newcommand{\vu}{\mathbf{u}}
\newcommand{\vv}{\mathbf{v}}

\newcommand{\mA}{\mathbf{A}}
\newcommand{\mB}{\mathbf{B}}
\newcommand{\mD}{\mathbf{D}}
\newcommand{\mE}{\mathbf{E}}
\newcommand{\mH}{\mathbf{H}}
\newcommand{\mI}{\mathbf{I}}
\newcommand{\mJ}{\mathbf{J}}
\newcommand{\mK}{\mathbf{K}}
\newcommand{\mM}{\mathbf{M}}
\newcommand{\mN}{\mathbf{N}}
\newcommand{\vN}[1]{\mathbf{N}_{#1}}
\newcommand{\mQ}{\mathbf{Q}}
\newcommand{\mRobsd}{\mathbf{R}}
\newcommand{\mUobsd}{\mathbf{U}}
\newcommand{\mUobsdperp}{\mUobsd_{\perp}}
\newcommand{\mV}{\mathbf{V}}
\newcommand{\mW}{\mathbf{W}}
\newcommand{\mWtilde}{\mathbf{\widetilde{W}}}
\newcommand{\mXobsd}{{\mathbf{\hat{X}}}}
\newcommand{\mY}{\mathbf{Y}}
\newcommand{\mZobsd}{\mathbf{Z}}
\newcommand{\mZobsdtilde}{\mathbf{\widetilde{Z}}}
\newcommand{\vZobsd}[1]{{\mathbf{Z}_{#1}}}

\newcommand{\mLambda}{\mathbf{\Lambda}}

\newcommand{\Ftrue}{F^\star}
\newcommand{\mFtrue}{\mathbf{F^\star}}
\newcommand{\Rtrue}{R^\star}
\newcommand{\mRtrue}{\mathbf{R^\star}}
\newcommand{\mUtrue}{{\mathbf{U^\star}}}
\newcommand{\Utruevec}[1]{{U^\star_{#1}}}
\newcommand{\mXtrue}{{\mathbf{X^\star}}}
\newcommand{\Ztrue}{Z^\star}
\newcommand{\mZtrue}{{\mathbf{Z^\star}}}
\newcommand{\vZtrue}[1]{{\mathbf{Z}^\star_{#1}}}
\newcommand{\mZtildestar}{{\mathbf{\widetilde{Z}^\star}}}
\newcommand{\mZtildestartop}{{\mathbf{\widetilde{Z}^{\star \top}}}}
\newcommand{\mZtruetilde}{\mZtildestar}
\newcommand{\mZtruetildetop}{\mZtildestartop}
\newcommand{\vZtruetilde}[1]{{\mathbf{\widetilde{Z}}^\star_{#1}}}
\newcommand{\Ztruetilde}{\widetilde{Z}^\star}
\newcommand{\mLambdatrue}{{\mLambda^\star}}
\newcommand{\mSigmatrue}{\mathbf{\Sigma^\star}}
\newcommand{\sigmatrue}[1]{{\sigma^\star_{#1}}}
\newcommand{\mDelta}{\mathbf{\Delta}}
\newcommand{\mOmega}{\mathbf{\Omega}}

\newcommand{\lambdatrue}[1]{{\lambda^\star_{#1}}}
\newcommand{\lambdaobsd}[1]{{\lambda_{#1}}}

\newcommand{\HH}{\text{H}}

\newcommand{\conjugate}[1]{\overline{#1}}

\begin{document}

\title{Adjacency Spectral Embeddings of Correlation Networks}
\author{Keith Levin}
\affil{\small University of Wisconsin--Madison}
\date{}

\maketitle

\begin{abstract}
In many applications, weighted networks are constructed based on time series data: each time series is associated to a vertex and edge weights are given by pairwise correlations.
The result is a network whose edge dependency structure violates the assumptions of most common network models.
Nonetheless, it is common to analyze these ``correlation networks'' using embedding methods derived from edge-independent network models, based on a belief that the edges are approximately independent.
In this work, we put this modeling choice on firm theoretical ground.
We show that when the time series are expressible in terms of a small number of Fourier basis elements (or in some other suitably-chosen basis), correlation networks correspond to latent space networks with dependent edge noise in which the vertex-level latent variables encode the basis coefficients.
Further, we show that when time series are observed subject to noise, spectral embedding of the resulting noisy correlation network still recovers these true vertex-level latent representations under suitable assumptions.
This characterization of embeddings as learning Fourier coefficients appears to be folklore in the signal processing community in the context of principal component analysis, but is, to the best of our knowledge, new to the statistical network analysis literature.
\end{abstract}

\section{Introduction} \label{sec:intro}

Networks have emerged as crucial data sources in scientific disciplines including neuroscience \citep{Sporns2012,connectal}, economics \citep{GolJac2014,TumLilMan2010}, microbiomics \citep{fierer_embracing_2017,jiang_microbiome_2019} and sociology \citep{Lazega2001,zhu2020multivariate} to name but a few.
As a result, a handful of workhorse statistical network models have emerged.
For example, the stochastic blockmodel \citep[SBM;][]{HolLasLei1983,Abbe2018} posits that each vertex belongs to one of $d$ communities, and the probability that two vertices form an edge is fully determined by their community memberships, with edges generated independently conditional on community memberships.
In recent years, the SBM has been extended in a variety of ways \citep[see, e.g.,][]{DCSBM,MMSBM,PABM,GolZheFieAir2010,GaoMaZhaZho2018,ZhaLevZhu2020,Abbe2018,JinKeLuo2024}.
For example, the degree-corrected SBM \citep{DCSBM} modifies the SBM to give each vertex a ``popularity'' parameter, so that vertices from the same community may exhibit distinct behaviors in their propensity to form edges.
Models such as the random dot product graph \citep[RDPG;][]{YouSch2007,RDPGsurvey,GRDPG},
graph root distribution \citep{Lei2021graphon},
Hoff-style models \citep{HofRafHan2002},
and the graphon \citep{Lovasz2012,TanCap2025}
further extend this idea, modeling network structure as driven by latent vertex-level vectors.
All of these models and their attendant estimation procedures rely on the assumption that edges are generated independently conditional on latent vertex-level structure.

In many applications, these edge independence assumptions are grossly violated.
Most notably, 
weighted networks are often derived from correlations between time series: a time series or other sequence is associated to each vertex, and edge weights are given by the correlations between these sequences.
In these {\em correlation networks}, edge independence assumptions clearly fail to hold: the same time series appears in all edge weights incident upon a given vertex.
As an example, consider connectomic networks \citep{Sporns2012}, in which vertices correspond to brain regions, and edge weights encode the strength of interaction between brain regions.
These interaction strengths are typically measured by correlations between blood oxygen levels in these brain regions \citep{MahTooBerMacBas2021,Vogelstein2025Pearson}. 
Thus, many connectomic networks are weighted \citep[or are binarized versions of weighted networks; see][and citations therein]{LeLevLev2018,MasBoyGarMuc2025}, with edges given by correlations between noisily-measured underlying signals from different brain regions.
As another example, consider affinity matrices, which encode networks whose vertices are stocks or other assets, and edges describe similarities of those assets \citep{FanFoutzJamesJank2014,hong2019inferring,FanKeLia2021}, usually measured by correlation across some set (e.g., correlation of prices or similarity of internet search activity).
As yet another example, in microbiome studies, it is common to construct networks whose vertices correspond to species of microorganisms, and edges capture the correlation of abundance counts across a collection of patients or environments \citep{lee_heterogeneous_2020,BeckerETAL2023,ZitLiWelGlaMorKri2024,BarSahBuf2025}.

In all of these example settings, the networks under study have weighted edges, with weights given by correlations between sequences.
For a given vertex, all edges incident on that vertex are based on the same sequence's correlations with all other observed sequences.
For example, in connectomic networks, the blood oxygen levels measured at brain region $i$ are used to compute the edge weights between region $i$ and all other brain regions.
The result is that edges in this network are clearly not independent, even after conditioning on any obvious choice of latent structure.
Despite this gap between model assumptions and reality, independent-edge latent variable models have been applied to these and related application domains with clear success. 
The apparent consensus in the network analysis community is that this violation of independence assumptions is not a concern. 
The goal of the present work is to put this folklore on firm ground by studying the properties of these correlation networks.
In particular, we focus here on the behavior of {\em embeddings} of correlation networks.

Embeddings, whereby vertices of a network are represented in a low-dimensional Euclidean space \citep{SusTanFisPri2012,SusTanPri2014,TanPri2018,node2vec}, are widely deployed for analysis and visualization of network data.
Embedding methods, by and large, bring with them implicit modeling assumptions.
For example, the adjacency spectral embedding \citep[ASE;][]{SusTanFisPri2012} implicitly assumes that the observed symmetric adjacency matrix $\mA \in \R^{n \times n}$ can be written as a low-rank matrix subject to noise,
\begin{equation} \label{eq:RDPG}
\mA = \mXtrue \mXtrue^\top + \mE,
\end{equation}
where the rows of $\mXtrue \in \R^{n \times d}$ encode ``latent positions'' of the $n$ vertices, and the matrix $\mE \in \R^{n \times n}$ has mean-zero entries independent (up to symmetry) given $\mXtrue$.
Similar ideas appear in, for example, \cite{GRDPG,LevAthTanLyzYouPri2017,Lei2021graphon,LevLodLev2022,GalJonBerPriRub2024}.
See \cite{RDPGsurvey} for discussion of the RDPG, the model that motivates the ASE and related spectral embeddings.
More generally, embeddings arise naturally in community detection as an input to $k$-means or other clustering methods to recover community memberships in the stochastic blockmodel \citep{vonLuxburg2007,RohChaYu2011,LyzSusTanAthPri2014,TwoTruths,LyzTanAthParPri2017}.

In light of the conditional independence assumptions on $\mE$, applying spectral embedding methods to correlation matrices violates what are otherwise standard, albeit unrealistic, network modeling assumptions.
Nonetheless, embeddings are used to analyze correlation or covariance networks throughout the sciences.
The ``folklore'' belief is that edge dependence can be safely ignored because the correlations are approximately independent.
An additional concern relates to the fact that network embeddings in edge-independent models can typically estimate a sensibly-defined latent structure.
For example, under suitable assumptions on $\mE$ in Equation~\eqref{eq:RDPG}, the ASE of $\mA$ recovers $\mXtrue$, up to orthogonal non-identifiability (see Section~\ref{sec:results} for further discussion).
When one applies the ASE or some other embedding method to a correlation network, there is no obvious analogue of the low-rank structure $\mXtrue$ in Equation~\eqref{eq:RDPG} for the embedding to estimate, nor is there an obvious interpretation for such a structure.

This work sets out to address these concerns and put the ``folklore'' on firm ground.
First, we show that correlation networks have a natural analogue of the low-dimensional latent structure $\mXtrue$ in Equation~\eqref{eq:RDPG}, which encode the Fourier coefficients of the time series (or the coefficients under any other suitably-chosen transform).
This is made formal in Lemmas~\ref{lem:FsqrtK:real} and~\ref{lem:ASElearnedFourier} below.
Second, we show in Theorem~\ref{thm:XobsdXtrue:tti} that when a collection of time series is observed subject to noise, the ASE applied to the correlation matrix recovers this low-dimensional encoding of the Fourier coefficients, up to non-identifiability constraints, at a rate that is qualitatively similar to those established for other related estimation problems.

{\bf Notation.} For a positive integer $n$, we write $[n] = \{1,2,\dots,n\}$.
Throughout, matrices and vectors are bolded, while scalars are in normal typeface. 
For a matrix $\mB$, we use $\mB^\top$ to denote its transpose and $\mB^\HH$ to denote its conjugate transpose.
We write $\|\mB\|$ for the spectral norm of $\mB$, $\| \mB \|_F$ for its Frobenius norm, and $\| \mB \|_{\tti}$ for its $(\tti)$-norm, i.e., the maximum Euclidean norm of the rows of $\mB$.
We denote the identity matrix by $\mI$, the matrix of all ones by $\mJ$ and the matrix of all zeros by $\mZero$, with dimensions made clear by context.
We use $\imag$ to denote the imaginary unit.
Note that we use the similar looking $i$ in places, but only ever as a summation index or subscript, so there is no risk of confusion.
For a complex number $z \in \C$, we use $\Re(z)$ and $\Im(z)$ to denote its real and imaginary components, respectively, so that $z = \Re(z) + \imag \Im(z)$.
We write $\bar{z}$ for the complex conjugate of $z$.

\section{Model Setup and Main Results} \label{sec:results}

Consider a collection of $n$ length-$T$ discrete-time time series, so that for each $i \in [n]$ and $t \in [T]$, $\Ztrue_{i,t} \in \R$ is the value of the $i$-th time series at time $t \in [T]$.
We write $\vZtrue{i} \in \R^T$ to denote the $i$-th time series, written as a $T$-dimensional vector.
We model these time series as subject to noise, so that our observed data takes the form $\vZobsd{i} = \vZtrue{i} + \vN{i}$ for each $i \in [n]$, where $\vN{i} \in \R^T$ is a vector of measurement noise (precise noise assumptions are given in Section~\ref{subsec:recoveringLatentStructure}).
For example, in our neuroimaging example discussed in Section~\ref{sec:intro}, the true signal $\vZtrue{i}$ corresponds to underlying brain activity in the brain region indexed by $i \in [n]$, while the noisy signal $\vZobsd{i}$ corresponds to the measured blood oxygen level in brain region $i$.
Collecting the true signals $\vZtrue{i}$ into the rows of $\mZtrue \in \R^{n \times T}$ and forming matrices $\mZobsd \in \R^{n \times T}$ and $\mN \in \R^{n \times T}$ analogously,
we may write
\begin{equation} \label{eq:def:Zobsd}
\mZobsd = \mZtrue + \mN \in \R^{n \times T} .
\end{equation}

In the applications discussed in Section~\ref{sec:intro}, it is common to form a weighted network with vertices $i\in [n]$ by assigning edge weights via correlation among these time series.
That is, ignoring measurement noise for now, we form a weighted network with adjacency matrix
\begin{equation} \label{eq:def:Rtrue}
\mRtrue
= \left[ \Rtrue_{i,j} \right]
= \left[ \Corr( \vZtrue{i}, \vZtrue{j} ) \right] 
= \left[ \frac{ \Cov( \vZtrue{i}, \vZtrue{j} ) }{ \sigmatrue{i} \sigmatrue{j} } \right]
\in \R^{n \times n},
\end{equation}
where 
\begin{equation} \label{eq:def:sigmatrue}
\sigmatrue{i}^2 
= \frac{1}{T} \sum_{t=1}^T \left( \Ztrue_{i,t} - \frac{1}{T}\sum_{s=1}^T \Ztrue_{i,s} \right)^2 .
\end{equation}
We refer to a network with adjacency matrix as in Equation~\eqref{eq:def:Rtrue} as a {\em correlation network}, in light of the fact that its edge weights encode correlations between time series. 

\begin{remark} \label{rem:power} 
On first glance, it is tempting to interpret $\sigmatrue{i}^2$ as simply a variance.
However, since we interpret $\vZtrue{i} \in \R^T$ as encoding network structure, rather than random noise, $\sigmatrue{i}^2$ is better viewed as a measure of signal strength.
Indeed, $\sigmatrue{i}^2$ corresponds precisely to the power in the signal $\vZtrue{i}$ after centering about its time-average \citep{OppSch2009}.
\end{remark}

Defining the ``centering'' matrix
\begin{equation} \label{eq:def:M}
\mM = \frac{1}{T}\left( \mI - \mJ \right) \in \R^{T \times T} ,
\end{equation}
we may write $\mRtrue$ as a matrix product,
\begin{equation} \label{eq:def:Rtrue:mxprod}
\mRtrue 
= \frac{1}{T} \mSigmatrue^{-1/2} \mZtrue \mM \mZtrue^\top \mSigmatrue^{-1/2},
\end{equation}
where
\begin{equation} \label{eq:def:sigma}
\mSigmatrue = \diag\left( \sigmatrue{1}^2, \sigmatrue{2}^2, \dots, 
		\sigmatrue{n}^2 \right)
\in \R^{n \times n} .
\end{equation}

Of course, in practice, we do not observe the true signals in $\mZtrue$, but instead observe $\mZobsd$ as in Equation~\eqref{eq:def:Zobsd}, and our resulting correlation network has an adjacency matrix given by 
\begin{equation} \label{eq:def:Robsd}
\mRobsd
= \frac{1}{T} \mSigmaobsd^{-1/2} \mZobsd \mM \mZobsd^\top \mSigmaobsd^{-1/2},
\end{equation}
where
\begin{equation} \label{eq:def:Sigmaobsd}
\mSigmaobsd 
= \diag\left( \sigmaobsd{1}^2, \sigmaobsd{2}^2, \dots, \sigmaobsd{n}^2 \right)
\in \R^{n \times n}
\end{equation}
is the diagonal matrix of (noisy) variances,
\begin{equation} \label{eq:def:sigmaobsd}
\sigmaobsd{i}^2
= \frac{1}{T} \sum_{t=1}^T \left( \Zobsd_{i,t} 
				- \frac{1}{T} \sum_{s=1}^T \Zobsd_{i,s} \right)^2.
\end{equation}

In the applications in Section~\ref{sec:intro}, one observes a collection of time series and forms a correlation network with adjacency matrix as in Equation~\eqref{eq:def:Rtrue} or~\eqref{eq:def:Robsd}.
We then apply an embedding method to this network.
Our goal in the present work is to address two questions:
\begin{enumerate}
\item Is there a sense in which $\mRtrue$ is driven by a low-dimensional latent structure, analogous to the role played by $\mXtrue$ in Equation~\eqref{eq:RDPG}?
\item Under what conditions does an embedding of the noisy correlation network $\mRobsd$ in Equation~\eqref{eq:def:Robsd} recover this latent structure, and at what rate?
\end{enumerate}

Below, we address both of these questions as they pertain to the adjacency spectral embedding.
Given a network with adjacency matrix $\mA \in \R^{n \times n}$, the $d$-dimensional ASE produces an $n$-by-$d$ matrix whose rows encode $d$-dimensional embeddings of the vertices:
\begin{equation} \label{eq:def:ASE}
\ASE( \mA, d ) = \mV ~|\mD|^{1/2} \in \R^{n \times d}
\end{equation}
where $\mD \in \R^{d \times d}$ is the diagonal matrix of the $d$ largest-magnitude eigenvalues of $\mA$ and $\mV \in \R^{n \times d}$ contains as its columns the corresponding eigenvectors.
Under a variety of network models, $\mA$ can be expressed as in Equation~\eqref{eq:RDPG}, and $\mXtrue$ is naturally viewed as a low-rank structure underlying the observed network.
As such, $\mXtrue$ is a natural target for estimation and inference.
Of course, estimation of $\mXtrue$ is possible only up to right-multiplication by an orthogonal matrix, since for any $\mQ \in \bbO_d$,
\begin{equation*}
\mXtrue \mXtrue^\top = \mXtrue \mQ (\mXtrue \mQ)^\top. 
\end{equation*}
Under suitable assumptions on $\mXtrue$ and $\mE$ in Equation~\eqref{eq:RDPG} \citep[see, e.g.,][]{LyzSusTanAthPri2014,RDPGsurvey,GRDPG,RohChaYu2011,LevLodLev2022,GalJonBerPriRub2024},
\begin{equation*}
\min_{\mQ \in \bbO_d} \left\| \ASE( \mA , d ) \mQ - \mXtrue \right\|_{\tti}
= \op{ 1 } .
\end{equation*}
That is, the ASE recovers the latent structure $\mXtrue$, up to orthogonal non-identifiability, uniformly over the vertices $i \in [n]$.
Our goal is to identify a low-dimensional structure analogous to $\mXtrue$ that gives rise to $\mRtrue$, and to show that the ASE of $\mRtrue$ recovers this structure (up to non-identifiability constraints) under suitable assumptions.
We focus here on the ASE owing to its comparative simplicity, but we anticipate that much of our analysis can be extended to other spectral embeddings and related methods \citep[e.g.,][]{TanPri2018,ModRub2021,GRDPG,node2vec,LinSusIsh2023,SheWanPri2023,ZhaTan2024}.

\subsection{Latent structure in correlation networks}

To understand the low-dimensional structure encoded in $\mRtrue$, we will aim to mimic the construction in Equation~\eqref{eq:RDPG} and write $\mRtrue = \mXtrue \mXtrue^\top$ for some choice of $\mXtrue \in \R^{n \times d}$, where $d \ll n$.
Toward this end, define
\begin{equation} \label{eq:def:Ztruetilde}
\mZtruetilde = \frac{1}{\sqrt{T}} \mSigmatrue^{-1/2} \mZtrue \mM \in \R^{n \times T},
\end{equation}
noting that the $i$-th row of this matrix, $\vZtruetilde{i} \in \R^T$, corresponds to a centered, normalized version of the $i$-th observed time series $\vZtrue{i} \in \R^T$.
Since $\mM$ is idempotent, we may rewrite Equation~\eqref{eq:def:Rtrue} as
\begin{equation} \label{eq:R:embeddingfactorization}
\mRtrue = \left( \frac{1}{\sqrt{T}} \mSigmatrue^{-1/2} \mZtrue \mM \right)
			\left( \frac{1}{\sqrt{T}} \mSigmatrue^{-1/2} \mZtrue \mM \right)^\top
	= \mZtruetilde \mZtruetildetop .
\end{equation}

Applying the discrete Fourier transform, we can express, for any $i \in [n]$ and $t \in [T]$,
\begin{equation} \label{eq:Z:IDFT}
\Ztruetilde_{i,t} = \frac{1}{T} \sum_{k=1}^{T}
        \Ftrue_{i,k} \exp\left\{ \frac{ 2 \pi \imag }{ T } (k-1)(t-1) \right\},
\end{equation}
where $\Ftrue_{i,k}$ denotes the $(k-1)$-th Fourier coefficient of the $i$-th standardized time series $\vZtruetilde{i}$ for each $k \in [T]$.
That is,
\begin{equation} \label{eq:F:DFT}
\Ftrue_{i,k} = \sum_{t=1}^{T} \Ztruetilde_{i,t} \exp\left\{ \frac{- 2 \pi \imag}{T} (k-1)(t-1) \right\}.
\end{equation}
We note that under this convention, $\Ftrue_{i,1}$ is the Fourier coefficient of $\mZtruetilde_i$ associated to frequency zero, which is $T$ times its mean.
Since $\mZtruetilde_i$ is standardized, we have $\Ftrue_{i,1} = 0$.

Denote by $\mOmega \in \C^{T \times T}$ the (unitary) inverse discrete Fourier transform matrix, with entries
\begin{equation*}
\Omega_{k,t} = \frac{1}{\sqrt{T}}\exp\left\{ \frac{2 \pi \imag}{T} (k-1)(t-1) \right\},
~k,t \in [T].
\end{equation*}
Examining Equation~\eqref{eq:Z:IDFT}, we see that we can express $\mZtruetilde$ as
\begin{equation*}
\mZtruetilde = \frac{1}{\sqrt{T} } \mFtrue \mOmega.
\end{equation*}
Substituting this into Equation~\eqref{eq:def:Rtrue:mxprod},
\begin{equation} \label{eq:Rtrue:Fourier}
\mRtrue = \frac{1}{\sqrt{T} } \mFtrue \mOmega
		\left( \frac{1}{\sqrt{T} } \mFtrue \mOmega \right)^\top
= \frac{1}{\sqrt{T} } \mFtrue \mOmega
		\left( \frac{1}{\sqrt{T} } \mFtrue \mOmega \right)^{\HH}
= \frac{1}{T} \mFtrue \mFtrue^{\HH},
\end{equation}
where the second equality follows from the fact that $\mZtrue$ is real
and the last equality follows from the fact that $\mOmega$ is unitary.

Since $\mFtrue$ encodes the Fourier transforms of real sequences, basic symmetry properties of the discrete Fourier transform \citep[][Chapter 8]{OppSch2009}, along with the fact that the rows of $\mZtruetilde$ are centered, imply 
\begin{equation} \label{eq:fourierSymmetry}
\overline{\Ftrue_{i,k}} = 
\begin{cases}
\Ftrue_{i,k} = 0
& \mbox{ if } k=1 \\
\Ftrue_{i,T-k+2} & \mbox{ if } k \in \{2,3,\dots,T\} .
\end{cases}
\end{equation}
Define a matrix $\mK \in \R^{T \times T}$ according to
\begin{equation} \label{eq:def:K}
K_{s,t} = \begin{cases} 1 &\mbox{ if } s=t=1 \\
	1 &\mbox{ if } s + t = T+2 \\
	0 &\mbox{ otherwise, }
\end{cases}
\end{equation}
and observe that by Equation~\eqref{eq:fourierSymmetry}, $\mFtrue^\HH = \left( \mFtrue \mK \right)^\top$. 
Substituting this into Equation~\eqref{eq:Rtrue:Fourier} and observing that $\mK$ is a permutation matrix and thus has a square root $\mK^{1/2} \in \C^{T \times T}$,
\begin{equation} \label{eq:Rtrue:Fourier:real}
\mRtrue = \frac{1}{T} \mFtrue \mK^{1/2} \left( \mFtrue \mK^{1/2} \right)^\top ,
\end{equation}
Basic properties of the Fourier transform yield the following result.
A detailed proof is given in Appendix~\ref{apx:fourier}.
\begin{lemma} \label{lem:FsqrtK:real}
If the matrix $\mZtrue$ has all real entries, then so does the matrix $\mFtrue \mK^{1/2}$, and these entries are given by
\begin{equation} \label{eq:FsqrtK:entry}
\left[ \mFtrue \mK^{1/2} \right]_{i,k}
= \Re\left( \Ftrue_{i,k} \right) - \Im\left( \Ftrue_{i,k} \right) 
~~~i \in [n], k \in [T] .
\end{equation}
\end{lemma}

In light of Lemma~\ref{lem:FsqrtK:real} and Equation~\eqref{eq:Rtrue:Fourier:real}, we see that $\mRtrue$ is indeed encoded by a low-rank structure if $\mFtrue \mK^{1/2}$ is low rank.
Suppose that there exist at most $d_0$ non-zero columns of $\mFtrue$.
That is, suppose that there exists a collection of $d_0$ Fourier basis elements whose span contains all of the rows in $\mZtruetilde$.
Equations~\eqref{eq:fourierSymmetry} and~\eqref{eq:FsqrtK:entry} imply the existence of $\mW \in \bbO_T$ such that the $i$-th row of $\mFtrue \mK^{1/2} \mW$ encodes, up to scaling, the real and imaginary parts of the Fourier coefficients of $\mZtruetilde_i$, and $\mFtrue \mK^{1/2} \mW$ has at most $2d_0$ non-zero columns.
It follows that $\mRtrue$ has rank $d \le 2d_0$, since from Equation~\eqref{eq:Rtrue:Fourier:real}, we have
\begin{equation*}
\mRtrue 
= \frac{1}{T} \mFtrue \mK^{1/2} \left( \mFtrue \mK^{1/2} \mW \right)^\top
= \frac{1}{T} \mFtrue \mK^{1/2} \mW \left( \mFtrue \mK^{1/2} \mW \right)^\top.
\end{equation*}
Thus, we may further write
\begin{equation} \label{eq:Rtrue:lowRank}
\mRtrue 
= \mUtrue \mLambdatrue \mUtrue^\top,
\end{equation}
where $\mLambdatrue = \diag( \lambdatrue{1},\lambdatrue{2},\dots,\lambdatrue{d})$ contains the $d \le 2d_0$ non-zero eigenvalues of $\mRtrue$ on its diagonal and $\mUtrue \in \R^{n \times d}$ has the $d$ corresponding orthonormal eigenvectors as its columns.
It follows that we may define
\begin{equation} \label{eq:def:Xtrue}
\mXtrue = \mUtrue \mLambdatrue^{1/2} \in \R^{n \times d}
\end{equation}
to be the ``population embedding'' of our (centered, scaled) time series data $\mZtruetilde \in \R^{n \times T}$ based on the correlation matrix $\mRtrue$, analogous to the low-rank matrix in Equation~\eqref{eq:RDPG}.
Since
\begin{equation*}
\mXtrue \mXtrue^\top = \mUtrue \mLambdatrue \mUtrue^\top = \mRtrue
= \mZtruetilde \mZtruetildetop
= \left( \frac{1}{\sqrt{T}} \mFtrue \mK^{1/2} \right)
    \left( \frac{1}{\sqrt{T}} \mFtrue \mK^{1/2} \right)^\top 
\end{equation*}
and all entries of $\mFtrue \mK^{1/2}$ are real,
there must exist $\mWtilde \in \bbO_T$ such that
\begin{equation*}
\begin{bmatrix} \mXtrue & \mZero \end{bmatrix} 
= \frac{1}{\sqrt{T}} \mFtrue \mK^{1/2} \mWtilde .
\end{equation*}
We have shown the following.

\begin{lemma} \label{lem:ASElearnedFourier} 
If the standardized time series in the rows of $\mZtruetilde \in \R^{n \times T}$ can be represented in a $d_0$-dimensional Fourier basis,
then there exists a $d \le 2d_0$ and $\mWtilde \in \bbO_T$ such that
\begin{equation*}
\begin{bmatrix} \ASE( \mRtrue, d ) & \mZero \end{bmatrix}
= \frac{1}{\sqrt{T}} \mFtrue \mK^{1/2} \mWtilde .
\end{equation*} 
That is, the $d$-dimensional ASE of $\mRtrue$ corresponds precisely to the Fourier representation of the signals in $\mZtruetilde$, up to orthogonal non-identifiability. 
\end{lemma}

Said another way, Lemma~\ref{lem:ASElearnedFourier} states that the Fourier coefficients of the signals in $\mZtruetilde$ correspond to a low-rank latent representation of the vertices in the correlation network $\mRtrue$.

\begin{remark} \label{rem:otherTransforms}
While our discussion above focused on representing the standardized time series $\mZtruetilde$ in the Fourier basis, the same basic argument applies to any basis that we might consider (e.g., wavelets or the discrete cosine transform). 
While we restrict our attention to the Fourier basis for the sake of concreteness, our results and analysis are largely agnostic to this choice of basis.
Key to our results is that there exist some unitary matrix $\mOmega$ such that the rows of $\mZtruetilde \mOmega^\HH$ encode the observed time series in the chosen basis. 
\end{remark}

\subsection{Recovering low-dimensional latent structure} \label{subsec:recoveringLatentStructure}

Lemmas~\ref{lem:FsqrtK:real} and~\ref{lem:ASElearnedFourier} indicate that the noiseless correlation network $\mRtrue$ has a low-rank latent structure given by Fourier transforms of the time series in the rows of $\mZtrue$.
We turn our attention now to the matter of recovering this latent structure from $\mRobsd$, the noisy version of $\mRtrue$, as given in Equation~\eqref{eq:def:Robsd}.
Following previous work on network embeddings \citep{SusTanFisPri2012,LevAthTanLyzYouPri2017,TanPri2018,LevRooMahPri2021}, we wish to study the adjacency spectral embedding of the observed correlation matrix $\mRobsd$ as defined in Equation~\eqref{eq:def:Robsd}.
In particular, write
\begin{equation} \label{eq:Robsd:decomp}
\mRobsd 
= \mUobsd \mLambdaobsd \mUobsd^\top 
	+ \mUobsdperp \mLambdaobsdperp \mUobsdperp^\top,
\end{equation}
where $\mLambdaobsd = \diag(\lambdaobsd{1},\lambdaobsd{2},\dots,\lambdaobsd{d})$ contains the $d$ largest eigenvalues of $\mRobsd$, $\mUobsd \in \R^{n \times d}$ contains the corresponding orthonormal eigenvectors as its columns, and $\mLambdaobsdperp = \diag( \lambdaobsd{d+1},\lambdaobsd{d+2},\dots,\lambdaobsd{n} )$ and $\mUobsdperp \in \R^{n \times (n-d)}$ are defined analogously for the trailing $n-d$ eigenvalues.
The $d$-dimensional ASE of the noisy correlation network is then given by
\begin{equation} \label{eq:def:mXobsd}
\mXobsd = \ASE( \mRobsd, d ) = \mUobsd \mLambdaobsd^{1/2} \in \R^{n \times d},
\end{equation}
and we may hope that if the noise is not too severe, $\mXobsd$ recovers $\mXtrue$ in Equation~\eqref{eq:def:Xtrue} up to non-identifiability. 

Of course, in practice, the dimension $d$ is unknown and must be selected from data.
This model selection task is well-studied in the networks literature under latent space models with conditionally independent edges \citep[see, e.g.,][]{NCV,LiLevZhu2020,HanYanFan2023,ChaSenChe2025,TaiLev2026}.
Myriad selection techniques for PCA and related dimensionality reduction methods have also been developed \citep{Joliffe2002,ZhuGho2006}, which may be applicable as well.
We leave a thorough exploration of this matter for future work, but Appendix~\ref{apx:experiments} includes an experimental investigation of the effects of model misspecification under the setting of Equation~\eqref{eq:def:Zobsd}.

Whether the ASE of $\mRobsd$ recovers the latent structure $\mXtrue$ depend, of course, on the behavior of the noise $\mN = \mZobsd - \mZtrue$.
We assume comparatively little about the specific structure of $\mN$ beyond row-wise independence.
In particular, we allow for dependence within the rows of $\mN$, subject to tail decay conditions.

\begin{assumption} \label{assum:N}
The rows $\mN_1,\mN_2,\dots,\mN_n \in \R^T$ of $\mN$ are mean-zero and independent, with
\begin{equation} \label{eq:assum:nT}
	n = O( T ) .
\end{equation}
For each $i \in [n]$, $\mN_i$ is a $\nu_i$-subgaussian random vector \citep[][Chapter 3]{Vershynin2020}.
That is, for any fixed unit vector $\vu \in \R^T$, the random variable $\vu^\top \mN_i$ is a $\nu_i$-subgaussian random variable \citep{BLM2013,Vershynin2020}.
\end{assumption}

Viewing the subgaussian parameter $\nu_i$ as a proxy for the variance in the $i$-th row of $\mN$, our results require that the noise not overwhelm the signal carried in $\mZtrue$.
In particular, the threshold for this rate of growth is set by the spectrum of the signal matrix $\mRtrue$.
Defining
\begin{equation} \label{eq:def:SNR}
\SNR = \min_{i \in [n]} \frac{ \sigmatrue{i}^2 }{ \nu_i },
\end{equation}
our results require that the condition number $\kappa = \lambdatrue{1}/\lambdatrue{d}$ of $\mRtrue$ not grow too quickly with respect to $\SNR$ and that the smallest signal eigenvalue $\lambdatrue{d}$ is bounded away from zero.

\begin{assumption} \label{assum:Rspec}
The signal eigenvalues $\lambdatrue{1} \ge \lambdatrue{2} \ge \cdots \ge \lambdatrue{d} > 0$ of the true correlation matrix $\mRtrue$ are such that
\begin{equation} \label{eq:assum:lambdad:LB}
    \lambdatrue{d} = \Omegap{ 1 }
    \end{equation}
and the condition number $\kappa = \lambdatrue{1}/\lambdatrue{d}$ is such that
$\kappa = o(T)$ and
\begin{equation} \label{eq:assum:noisepower}
    \frac{ \kappa \log T }{ \SNR } 
    = o( 1 ) .
    \end{equation}
\end{assumption}

The following theorem characterizes how well the embeddings of $\mRobsd$ recover their population counterparts.
The result serves as a correlation network analogue of $(\tti)$-norm estimation error bounds found in, for example, \cite{SusTanFisPri2012,LyzSusTanAthPri2014,GRDPG,TanPri2018,LevLodLev2022}, which rely on edge independence assumptions.
A proof of this theorem can be found in Appendix~\ref{apx:thm:tti}.

\begin{theorem} \label{thm:XobsdXtrue:tti}
Under Assumptions~\ref{assum:N} and~\ref{assum:Rspec}, let $\mXtrue \in \R^{n \times d}$ be as in Equation~\eqref{eq:def:Xtrue}, and let $\mXobsd = \ASE( \mRobsd, d )$.
There exists $\mQ \in \bbO_d$ such that
\begin{equation*} \begin{aligned}
\min_{\mQ \in \bbO_d} \left\| \mXobsd - \mXtrue \mQ \right\|_{\tti}
\le
C \sqrt{ \frac{ d \log T }{ \SNR \lambdatrue{d} } }
\left[ 1
        + \kappa \left( \kappa \sqrt{\frac{ \log n }{ \SNR }} 
                + \sqrt{ \frac{ 1 }{ T } } \right) \| \mRtrue \|_{\tti} 
\right]
\end{aligned} \end{equation*} 
\end{theorem}

It is instructive to consider the upper bound in Theorem~\ref{thm:XobsdXtrue:tti} when $\kappa$ and $d$ are both bounded above by constants.
Using basic properties of the $(\tti)$-norm \citep{CaiZha2018,CapTanPri2019} and the fact that $\mZtruetilde$ has standardized rows, we have
\begin{equation*}
\| \mRtrue \|_{\tti} = \| \mZtruetilde \mZtruetildetop \|_{\tti}
\le \| \mZtruetilde \|_{\tti} \| \mLambdatrue^{1/2} \|
= \sqrt{ \lambdatrue{1} }.
\end{equation*}
It follows that when $\max\{ \kappa, d \} = O(1)$, we can rewrite the bound in Theorem~\ref{thm:XobsdXtrue:tti} as
\begin{equation} \label{eq:XobsdXtrue:tti:constantStuff}
\min_{\mQ \in \bbO_d} \left\| \mXobsd - \mXtrue \mQ \right\|_{\tti}
\le
C \sqrt{ \frac{ \log T }{ \SNR } }
\left( \sqrt{\frac{1}{\lambdatrue{d}}}
	+ \sqrt{\frac{ \log n }{ \SNR }} + \sqrt{ \frac{ 1 }{ T } } \right) ,
\end{equation}
and we see that consistent estimation occurs under Assumption~\ref{assum:Rspec}.

Among previous work, most similar to the setting of Theorem~\ref{thm:XobsdXtrue:tti} is the literature on heteroscedastic PCA, which considers a setting similar to Equation~\eqref{eq:def:Zobsd}, where one observes $\mZobsd = \mZtrue + \mN$, and aims to recover the left singular vectors of $\mZtrue$.
When the entries within a row of $\mN$ differ in their variances, as we allow in the present work, ``vanilla'' PCA fails due to bias in the diagonal entries of the empirical covariance matrix.
Methods for correcting this bias include zeroing out the diagonal of the empirical covariance before computing its eigenvectors \citep{AbbFanWan2022,CaiLiChiPooChe2021} and iterative approaches that seek to impute or otherwise debias these diagonal entries \citep{ZhaCaiWu2022,ZhoChe2025}.
Importantly, the estimation problem considered in Theorem~\ref{thm:XobsdXtrue:tti} differs from heteroscedastic PCA (and many other related subspace estimation tasks) in its target of inference: we aim to estimate the scaled leading eigenvectors of $\mRtrue$ (i.e., $\mXtrue$ in Theorem~\ref{thm:XobsdXtrue:tti}; see also Lemma~\ref{lem:ASElearnedFourier}), while PCA aims to estimate the left singular vectors of $\mZtrue$.
Given this difference, Theorem~\ref{thm:XobsdXtrue:tti} cannot be directly compared to results for heteroscedastic PCA.
Nonetheless, it is informative to compare our assumptions with those used in this related literature.

\cite{ZhaCaiWu2022} consider the setting where the rows of $\mZtrue$ are drawn i.i.d.~from a distribution with rank-$d$ covariance matrix $\mDelta_0 \in \R^{T \times T}$, and one observes $\mZobsd = \mZtrue + \mN$ where the rows of $\mN$ are i.i.d.~mean zero with independent entries, but these entries are allowed to differ in their variances.
The goal is to estimate the signal eigenvectors of $\mDelta_0$.
\cite{CaiLiChiPooChe2021,YanCheFan2024,ZhoChe2025}
consider a similar setting, with relaxed assumptions on the noise $\mN$ more in line with the tail bounds of our Assumption~\ref{assum:N}, but still requiring entrywise independence in $\mN$.
We note that \cite{CaiLiChiPooChe2021,YanCheFan2024} allow for missing entries in $\mZobsd$, which we do not consider here, though we expect that low-rank imputation methods as used in \cite{LiLevZhu2020} can be applied in the setting of Theorem~\ref{thm:XobsdXtrue:tti}.

Most similar in spirit to our assumptions are \cite{AbbFanWan2022,AgtLubPri2022}, both of which allow for dependence within rows of $\mN$ and impose tail bounds similar to our Assumption~\ref{assum:N}.
Both also require SNR-related assumptions that are superficially similar to our Assumption~\ref{assum:Rspec}, but are in fact largely incomparable.
Assumption~\ref{assum:Rspec} concerns $\SNR$, which measures how the signal present in individual rows of $\mZtrue$, as measured by $\sigmatrue{i}^2$, compare to the variance-like quantity $\nu_i$, which describes the tail behavior of the $i$-th row of the noise matrix $\mN$.
In contrast, \cite{AbbFanWan2022,AgtLubPri2022} measure SNR by comparing $\lambdatrue{d}$ to the variance of the noise (i.e., $\nu_1,\nu_2,\dots,\nu_n$ in our notation).
\cite{AbbFanWan2022} and \cite{AgtLubPri2022} also require assumptions on the coherence of the left singular vectors of $\mZtrue$.
It is interesting to note that Theorem~\ref{thm:XobsdXtrue:tti} requires no analogous assumptions on the eigenstructure of $\mRtrue$.
Nor does it require assumptions on how the covariances within the rows of $\mN$ interact with the signal eigenstructure, as required in \cite{AgtLubPri2022}.

\section{Experiments} \label{sec:expts}

We turn now to a brief exploration of our theoretical results.
In particular, we are interested in verifying the estimation bound in Theorem~\ref{thm:XobsdXtrue:tti}, which predicts the rate at which $\ASE( \mRobsd, d )$ should recover $\mXtrue$ once we account for orthogonal non-identifiability.
When evaluating an estimate $\mY \in \R^{n \times d}$ of $\mXtrue$, we will take advantage of the fact that, per Lemma~\ref{lem:ASElearnedFourier}, $\mXtrue$ and $\mZtruetilde$ are equal up to orthogonal rotation, and assess estimation error by
\begin{equation} \label{eq:procrustes}
\min_{\mW \in \bbO_T}
\left\| \begin{bmatrix} \mY & \mZero \end{bmatrix} \mW 
	- \mZtruetilde \right\|_{\tti},
\end{equation}
where we have appended $T-d$ columns of zeros to $\mY$ so that it is conformal with $\mZtruetilde$.

In what follows, we compare the ASE against two baseline methods for estimating $\mZtruetilde$.
The first is based on applying $d$-dimensional PCA to the centered matrix of time series $\mZobsd \mM \in \R^{n \times T}$ followed by standardizing the rows.
Note that we perform a scaled PCA, with the eigenvectors scaled by the square roots of the eigenvalues.
In essence, this PCA-based estimate reverses the steps of the ASE by applying eigenvalue truncation {\em before} performing row normalization.
Our second comparison is with a na\"ive method, whereby we simply use
\begin{equation*}
\mZobsdtilde = 
\frac{1}{\sqrt{T}} \mSigmaobsd^{-1/2} \mZobsd \mM \in \R^{n \times T},
\end{equation*}
as our estimate of $\mZtruetilde$.
While such an estimate fails to yield dimensionality reduction the way that the ASE or other PCA-based methods might, it serves as a baseline comparison to indicate whether an estimator is successfully removing noise from the observed signals.

\subsection{Effect of variance on estimation} \label{subsec:expt:variance}

We begin by examining how the behavior of the rows of $\mN$ influence estimation accuracy.
For a given number of vertices $n$, number of time series $T$, and a number of signal frequencies $d_0 < T/2$, we generate $\mZtrue$ by choosing $d_0$ Fourier basis element indices from the set $\{2,3,\dots,\lfloor T/2\rfloor\}$ uniformly at random without replacement.
Letting $\calK$ denote this index set, for each $i \in [n]$, we generate $\mZtrue_i$ by populating the real and imaginary components of the Fourier coefficients indexed by $\calK$  with i.i.d.~draws from a standard normal.
To ensure that $\mZtrue$ has all real elements, for each $k \in \calK$, we take the basis element with index $T-k+2$ to have Fourier coefficient equal to the complex conjugate of the coefficient for index $k$, per Equation~\eqref{eq:fourierSymmetry}
The result is that the $n$ rows of $\mZtrue \in \R^{n \times T}$ lie in the span of a $d_0$-dimensional Fourier basis, and $\mZtrue$ has rank either $2d_0$ or $2d_0-1$, depending on the parity of $T$ and, if $T$ is even, whether or not we have selected the basis element that is its own complex conjugate.
Finally, we rescale $\mZtrue$ so that $\| \mZtrue \|_F = \sqrt{n}$.
That is, $\mZtrue$ has average row norm equal to one.

Our upper bound in Theorem~\ref{thm:XobsdXtrue:tti} depends on the variance of the rows of $\mN$ via the parameter $\SNR$ defined in Equation~\eqref{eq:def:SNR}, which captures the minimum, over all $i \in [n]$, of the ratio of power in $\vZtrue{i} \in \R^T$ to the variance-like parameter $\nu_i$.
To explore the effect of $\SNR$, we generate each row $\mN_i$ of $\mN$ 
$i\in[n]$, by drawing the entries of $\mN_i$ i.i.d.~according to a mean zero normal with variance $\nu \| \vZtrue{i} \|^2$.
Under this setup, the ratio $\nu_i/\sigmatrue{i}^2$ is constant over all $i\in[n]$ and we have $\SNR=\nu$.
We evaluate the $(\tti)$-norm between our estimates and $\mZtruetilde$ after the best Procrustes alignment, as in Equation~\eqref{eq:procrustes}.
We repeated this procedure $50$ times for each combination of conditions.
The results of this experiment are displayed in Figure~\ref{fig:exactSNR:tti-by-var}, which shows $(\tti)$-norm error as a function of the variance parameter $\nu$ for several choices of $d_0$.
Examining the figure, we observe that the ASE out-performs both our ``na\"{i}ve'' baseline and the row-normalized PCA estimate, with the largest gap holding at lower noise levels (i.e., lower entrywise variance) and lower values of $d_0$.
Most importantly, we note that the behavior of the $(\tti)$-norm error with respect to the entrywise variance matches the rate predicted by Theorem~\ref{thm:XobsdXtrue:tti}.
Noting that $\SNR=1/\nu$ in the present setting, we see that estimation error decays approximately as the square root of $1/\SNR$ for all four choices of $d_0$.

\begin{figure} 
    \centering
    \includegraphics[width=\textwidth]{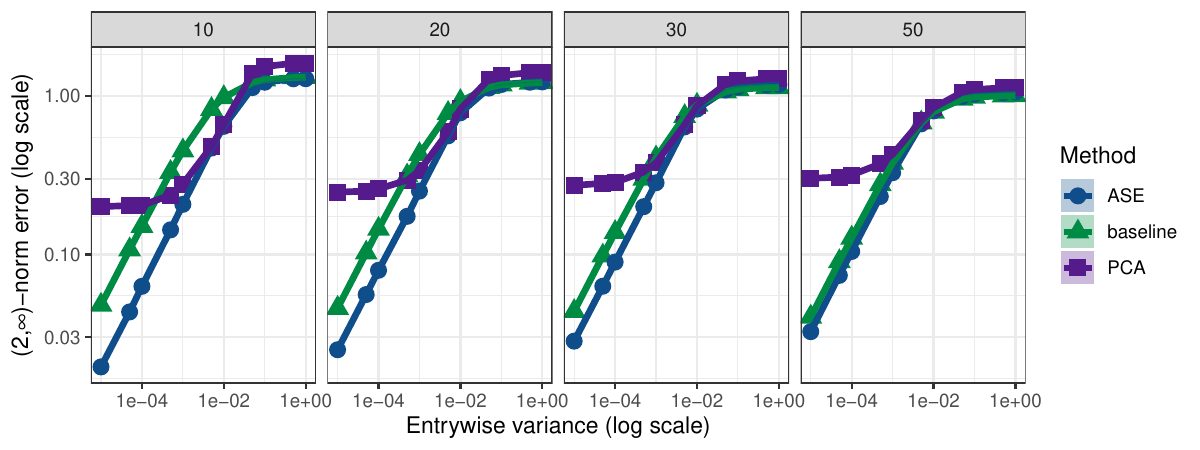}
    \caption{Estimation error in $(\tti)$-norm as a function of variance parameter $\nu$ by the ASE (blue circles), PCA (purple squares) and the na\"ive baseline (green triangles), when applied to $n=200$ time series of length $T=200$ for four choices of Fourier basis size $d_0 = 10,20,30,50$.}
    \label{fig:exactSNR:tti-by-var}
\end{figure}

The quantity $\SNR$ in Theorem~\ref{thm:XobsdXtrue:tti} is defined as a minimum over all $i \in [n]$ of an SNR-like quantity.
This raises the question as to whether the minimum in fact drives the convergence rate.
To assess this, we modify the experiment described above so that one randomly-chosen row $i_0 \in [n]$ of $\mN$ is generated with entries i.i.d.~from a mean zero normal with variance $\alpha \nu \| \vZtrue{i_0} \|^2$.
When $\alpha=1$, we recover the experiment from Figure~\ref{fig:exactSNR:tti-by-var}.
Increasing $\alpha$ the variance of the $i_0$-th row of $\mN$ compared to the rest, and decreases $\SNR$.
With $n=200,T=500$ and $d_0=20$, we performed this experiment for $\nu=10^{-6},10^{-5}$ and $10^{-4}$ and varying choices of $\alpha$, with $50$ Monte Carlo trials for each combination of conditions.
Figure~\ref{fig:heteroSNR:tti-by-inflation} summarizes the results, showing $(\tti)$-norm estimation error as a function of the inflation factor $\alpha$.
When $\nu$ is not too big (i.e., the left and middle plots), ASE estimation error increases as approximately the square root of $\alpha$, in line with the predictions of Theorem~\ref{thm:XobsdXtrue:tti}.
Examining the right-most subplot, which corresponds to higher entrywise variance $\nu$, we see that performance degrades nearly to the point of being indistinguishable from the na\"ive baseline, with the ASE and PCA-based method performing similarly poorly.
Additional experiments exploring the effect of Gaussian versus Laplacian noise are presented in Appendix~\ref{apx:experiments}.

\begin{figure} 
    \centering
    \includegraphics[width=\textwidth]{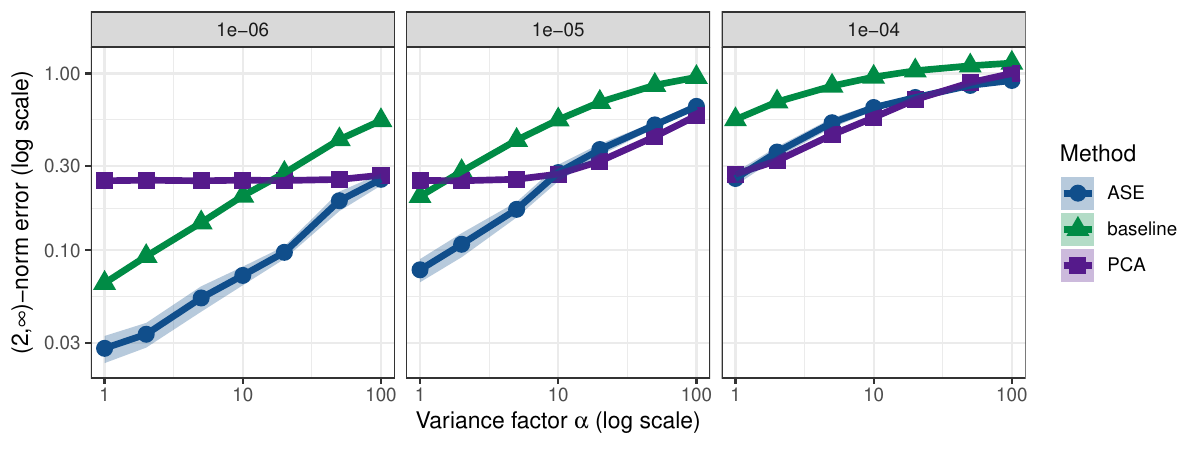}
    \caption{Estimation error in $(\tti)$-norm as a function of variance factor $\alpha$ by the ASE (blue circles), PCA (purple squares) and the na\"ive baseline (green triangles), as applied to $n=200$ time series of length $T=500$ with Fourier basis size $d_0 = 20$. Shaded regions indicate two standard errors of the mean.}
    \label{fig:heteroSNR:tti-by-inflation}
\end{figure}

\subsection{Effect of dimension} 

The rate in Theorem~\ref{thm:XobsdXtrue:tti} suggests that the ASE estimation error should also depend on the model dimension (i.e., twice the number of signal frequencies $d_0$), with estimation error growing as $\sqrt{d_0}$.
To investigate this, we continue with the same experimental setup as the experiment in Figure~\ref{fig:exactSNR:tti-by-var}, this time with fixed network size $n=1200$ and time series length $T=1800$, and varying the number of signal frequencies $d_0$.
For each combination of conditions, we performed 50 indepednet Monte Carlo trials and recorded the mean $(\tti)$-estimation error.
The results of this experiment is summarized in Figure~\ref{fig:exactSNR:tti-by-d0} for three different choices of noise variance $\nu$.
We see that at higher variances (i.e., $\nu=0.001$ on the right), none of our three estimation methods perform especially well, but at lower variances, the ASE outperforms both the PCA-based method and the na\"ive baseline, with estimation error obeying the $\sqrt{d}$-rate predicted by our theory.

\begin{figure} 
    \centering
    \includegraphics[width=\textwidth]{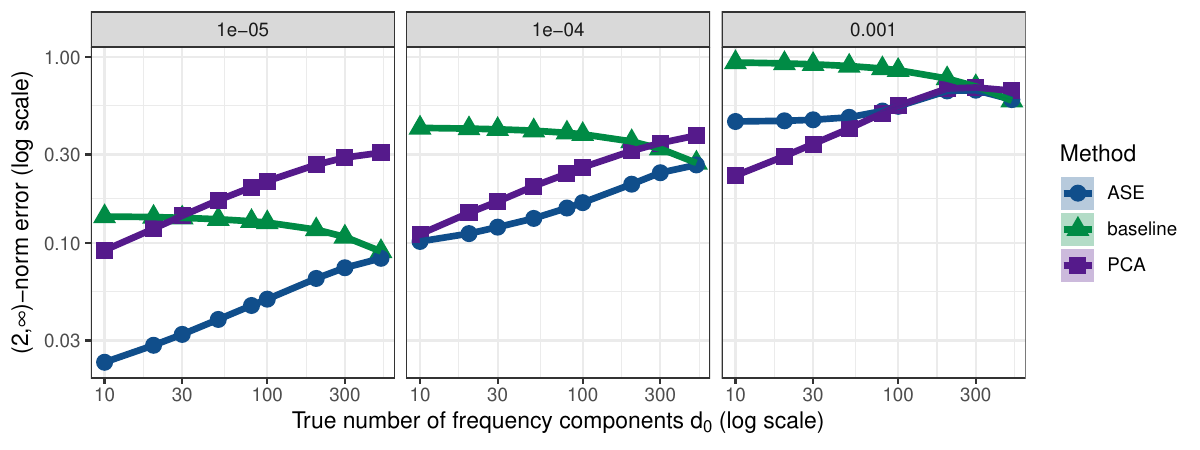}
    \caption{Estimation error in $(\tti)$-norm as a function of the number of frequencies $d_0$ for ASE (blue circles), PCA (purple squares) and na\"ive baseline (green triangles), when applied to $n=1200$ time series of length $T=1800$ under three different choices of noise variance $\nu$. }
    \label{fig:exactSNR:tti-by-d0}
\end{figure}

\subsection{Effect of spectrum of $\mRtrue$}

Examining the bound in Theorem~\ref{thm:XobsdXtrue:tti}, the estimation rate is influenced by $\lambdatrue{d}$, the smallest signal eigenvalue in $\mRtrue$.
Thus, we wish to conduct a numerical experiment similar to those above, this time varying $\lambdatrue{d}$ while keeping other model parameters fixed.
Ideally, this would include the condition number $\kappa$ of $\mRtrue$.
Unfortunately, the fact that $\mZtruetilde$ have unit-norm rows complicates such an experiment.
By construction, the rows of $\mZtruetilde$ are all unit norm, and by Lemma~\ref{lem:ASElearnedFourier}, we have $\| \mZtruetilde \|_F^2 = \sum_{k=1}^d \lambdatrue{k}$.
Thus, with $\lambdatrue{1}=\lambdatrue{2}=\cdots=\lambdatrue{d}$, so that $\kappa=1$, the number of vertices $n$ and the number of signal eigenvalues $d$ completely determines $\lambdatrue{d}$ (and all other eigenvalues of $\mRtrue$).
A further challenge arises from the fact that such an experiment requires that we generate $\mZtruetilde$ not only to have all $d$ singular values the same, but to have row sums equal to $0$ and all rows with the same norm.
Rather than trying to satisfy all of these constraints exactly, we generate the rows of $\mZtrue$ as in the previous two experiments, but standardize them so that $\mZtrue=\mZtruetilde$ 
In practice, so long as the number of signal frequencies $d_0$ is not large, this rarely yields a condition number larger than $5$, so that, at least for the purposes of asymptotics, we may treat $n/d$ as equal to $\lambdatrue{d}$.
Note that under this setting, for a fixed model rank $d$, increasing $n$ is equivalent to increasing $\lambdatrue{d}$.
Once we have generated $\mZtrue$, we draw the entries of the noise $\mN$ i.i.d.~from a mean zero normal with variance $\nu = 1e-4$ and construct $\mZobsd = \mZtrue + \mN$.
Note that owing to our normalization constraint on $\mZtrue$ in this experiment, this variance is not directly comparable to that used in our first two experiments.
The results of this experiment are displayed in Figure~\ref{fig:varylambda:tti-by-n}.
Examining the figure, we see that the behavior of the ASE is consistent with the Theorem~\ref{thm:XobsdXtrue:tti}.
When $\kappa$ is constant, Equation~\eqref{eq:XobsdXtrue:tti:constantStuff} shows that increasing $\lambdatrue{d}$ is not sufficient to ensure convergence of the ASE without also allowing $\SNR$ and $T$ to grow, both of which are held fixed in this experiment.
We note that this problem setting is artificially easy for the PCA-based method.
Since the rows of $\mZtrue$ are already normalized, the PCA-based method is, in essence, just a truncation of $\mZobsd$.
As such, it is not surprising that the PCA-based method converges at a $1/\sqrt{\lambdatrue{d}}$ rate in this experiment.

\begin{figure} 
    \centering
    \includegraphics[width=\textwidth]{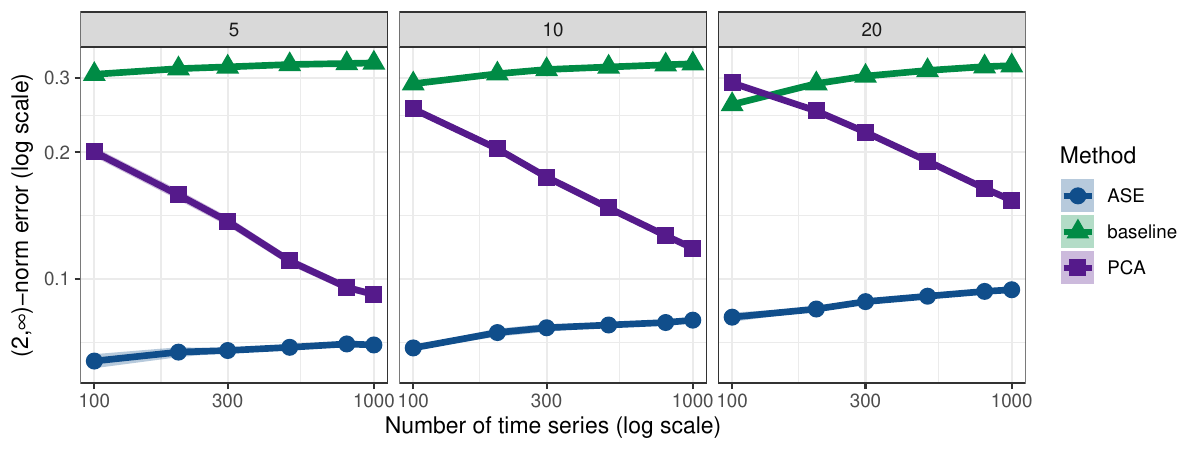}
    \caption{Estimation error in $(\tti)$-norm as a function of the number of time series $n$ by the ASE (blue circles), PCA (purple squares) and the na\"ive baseline (green triangles), as applied to time series of length $T=1000$ under three different choices of $d_0$.}
    \label{fig:varylambda:tti-by-n}
\end{figure}

\section{Discussion and Conclusion} \label{sec:conclusion}

Network embedding methods with implicit edge independence assumptions are frequently applied to correlation networks that clearly violate edge independence.
Practitioners nonetheless frequently apply these methods to correlation networks, under the intuition that the dependence among the edges is ignorable.
We have presented an analysis of the behavior of spectral embeddings as applied to correlation networks that puts this intuition on firm ground.
We have shown that correlation networks built from collections of time series have a clear interpretation as having latent positions encoding the Fourier coefficients of the time series, and we have shown that when noisily-observed time series are used to build a correlation network, the ASE provably recovers the true underlying low-rank structure.

There are several interesting directions for future work.
We conjecture that the rate in Theorem~\ref{thm:XobsdXtrue:tti} is minima optimal, up to logarithmic factors and dependence on the condition number $\kappa$.
Establishing a minimax lower bound is the focus of ongoing follow-up work, starting from the arguments developed in \cite{YanLev2023,YanLev2024_NeurIPS}.
We expect that the dependence on $\kappa$ can be relaxed, and perhaps removed entirely, with suitably careful bookkeeping \citep[see, e.g.,][]{Pensky2024}.
If dependence on $\kappa$ cannot be eliminated, future work should investigate whether approaches similar to \cite{ZhoChe2025} can be used as a preprocessing step to remove the effect of condition number.
As alluded to when comparing our results to the heteroscedastic PCA literature, it is interesting to note that Theorem~\ref{thm:XobsdXtrue:tti} does not suggest a bias introduced by heteroscedasticity in the rows of $\mN$. 
Future work should more carefully analyze how, if at all, heteroscedasticity contributes to the estimation rate in Theorem~\ref{thm:XobsdXtrue:tti} and whether, in the event that there is structure present in the eigenvectors of $\mRtrue$ (i.e., the left singular vectors of $\mZtruetilde$), this structure can be leveraged to achieve improved estimation, analogously to recent results in low-rank matrix estimation under independent noise \citep{YanLev2024_NeurIPS,YanLev2025_AISTATS}.
Finally, as discussed in Section~\ref{sec:results}, embeddings of correlation networks and PCA applied directly to $\mZtrue$ both provably learn low-rank representations of the time series.
Future work should characterize when one or the other of these two methods is preferable, analogous to the study in \citep{TwoTruths} comparing the ASE to the Laplacian spectral embedding \citep{TanPri2018}.

\paragraph{Acknowledgements.}
The author would like to thank
Jes\'{u}s Arroyo,
Avanti Athreya,
Nathan Aviles,
Joshua Cape,
Alex Hayes,
Elizaveta Levina,
Zachary Lubberts,
Vince Lyzinski,
Karl Rohe,
Joseph Salzer,
Roddy Taing,
Hao Yan,
Wen Zhou, 
Ji Zhu
and the attendees of the UW-Madison IFDS Seminar
for helpful discussions and suggestions that have greatly improved this paper.
The author was supported in part by the University of Wisconsin-Madison, Office of the Vice Chancellor for Research and Graduate Education with funding from the Wisconsin Alumni Research Foundation.

\bibliographystyle{plainnat}
\bibliography{biblio}

\newpage

\appendix

\section{Fourier Coefficients and Low-Rank Truncation of $\mRtrue$} \label{apx:fourier}

Here we prove Lemma~\ref{lem:FsqrtK:real}, which implies that we may interpret the low-rank truncation of a correlation matrix as capturing the Fourier coefficients of the time series, up to orthogonal rotation.
We remind the reader that $\mK \in \R^{T \times T}$ denotes the matrix defined in Equation~\eqref{eq:def:K}.

\begin{proof}[Proof of Lemma~\ref{lem:FsqrtK:real}]
We begin by noting that by construction, $\mK \ve_1 = \ve_1$, and for all $\ell \in \{2,3,\dots,T\}$,
$\mK \ve_\ell = \ve_{T-\ell+2}$.
It follows that $\mK$ has eigendecomposition
\begin{equation*}
\mK = \ve_1 \ve_1^\top
+ \sum_{\ell=2}^{\lceil (T-1)/2 \rceil} 
	\frac{ \left(\ve_\ell + \ve_{T-\ell+2}\right) \left(\ve_\ell + \ve_{T-\ell+2}\right)^\top }{ \| \ve_\ell + \ve_{T-\ell+2} }
	- \sum_{\ell=2}^{\lfloor (T-1)/2 \rfloor} 
		\frac{ \left(\ve_\ell - \ve_{T-\ell+2}\right) \left(\ve_\ell - \ve_{T-\ell+2}\right)^\top }{ \| \ve_\ell - \ve_{T-\ell+2} \| } ,
\end{equation*}
and thus we can construct the square root of $\mK$ according to
\begin{equation} \label{eq:sqrtK:decomp}
\begin{aligned}
\mK^{1/2}
&= \ve_1 \ve_1^\top
+ \sum_{\ell=2}^{\lceil (T-1)/2 \rceil} 
	\frac{ \left(\ve_\ell + \ve_{T-\ell+2}\right) 
				\left(\ve_\ell + \ve_{T-\ell+2}\right)^\top }
		{ \| \ve_\ell + \ve_{T-\ell+2} \| } \\
&~~~~~~~~~~~+ \sum_{\ell=2}^{\lfloor (T-1)/2 \rfloor}
	\frac{ \imag \left(\ve_\ell - \ve_{T-\ell+2}\right) 
				\left(\ve_\ell - \ve_{T-\ell+2}\right)^\top }
		{ \| \ve_\ell - \ve_{T-\ell+2} \| } .
\end{aligned} \end{equation}

Recall from Equation~\eqref{eq:F:DFT} that the rows of $\mFtrue$ encode the Fourier coefficients of the corresponding rows of $\mZtrue$, which is assumed to have all real entries.
Since the rows of $\mZtruetilde$ are centered by construction, we have
\begin{equation} \label{eq:DFT:constantTerm}
\Ftrue_{i,1} = \sum_{t=1}^{T} \Ztruetilde_{i,t} = 0
\end{equation} 
for all $i \in [n]$, so that
\begin{equation} \label{eq:e1}
\mFtrue \ve_1 \ve_1^\top = \mZero .
\end{equation} 
On the other hand, recalling Equation~\eqref{eq:fourierSymmetry}, for all $i \in [n]$ and $\ell \in \{2,3,\dots,T\}$, 
\begin{equation} \label{eq:conjugateSymmetry}
\Ftrue_{i,\ell} = \conjugate{\Ftrue_{i,T-\ell+2}} .
\end{equation}
It follows that for any $i \in [n]$ and any $\ell \in \{2,3,\dots,\lfloor (T-1)/2 \rfloor\}$,
\begin{equation*} 
\left[ \left(\mFtrue + \conjugate{\mFtrue}\right) \left(\ve_\ell - \ve_{T-\ell+2}\right) \right]_i
= \Ftrue_{i,\ell} - \Ftrue_{i,T-\ell+2} + \conjugate{\Ftrue_{i,\ell}} - \conjugate{\Ftrue_{i,T-\ell+2}}
= 0 ,
\end{equation*}
and therefore, using the decomposition in Equation~\eqref{eq:sqrtK:decomp} and recalling Equation~\eqref{eq:e1},
\begin{equation} \label{eq:realFsqrtK}
\frac{\mFtrue + \conjugate{\mFtrue}}{2} \mK^{1/2}
= \frac{1}{2} \sum_{\ell=2}^{\lceil (T-1)/2 \rceil} 
	\left( \frac{ \mFtrue + \conjugate{\mFtrue} }{2} \right)
		\left( \ve_\ell + \ve_{T-\ell} \right)
		\left( \ve_\ell + \ve_{T-\ell} \right)^\top
\in \R^{n \times T}.
\end{equation}
We note that this quantity is purely real, since $(\mFtrue + \conjugate{\mFtrue})/2$ is the real part of $\mFtrue$.

Similarly, for any $i \in [n]$ and any $\ell \in \{2,3,\dots,\lceil (T-1)/2 \rceil\}$,
\begin{equation*} 
\left[ \left(\mFtrue - \conjugate{\mFtrue}\right) \left(\ve_\ell + \ve_{T-\ell+2}\right) \right]_i
= \Ftrue_{i,\ell} + \Ftrue_{i,T-\ell+2} 
	- \left(\conjugate{\Ftrue_{i,\ell}} + \conjugate{\Ftrue_{i,T-\ell+2}} \right)
= 0,
\end{equation*}
from which
\begin{equation} \label{eq:imagFsqrtK}
\frac{\mFtrue - \conjugate{\mFtrue}}{2} \mK^{1/2}
= \frac{\imag}{2} \sum_{\ell=2}^{\lfloor (T-1)/2 \rfloor}
	\left( \frac{ \mFtrue - \conjugate{\mFtrue} }{ 2 } \right)\left(\ve_\ell - \ve_{T-\ell+2}\right)
					\left(\ve_\ell - \ve_{T-\ell+2}\right)^\top,
\end{equation}
which is purely real, since $\mFtrue - \conjugate{\mFtrue}$ is purely imaginary.
Decomposing $\mFtrue$ into its real and imaginary parts as
\begin{equation*}
\mFtrue = \frac{\mFtrue + \conjugate{\mFtrue}}{2} + \frac{\mFtrue - \conjugate{\mFtrue}}{2},
\end{equation*}
Equations~\eqref{eq:realFsqrtK} and~\eqref{eq:imagFsqrtK} imply that
\begin{equation*}
\mFtrue \mK^{1/2}
= \left( \frac{\mFtrue + \conjugate{\mFtrue}}{2} \right) \mK^{1/2}
+ \left( \frac{\mFtrue - \conjugate{\mFtrue}}{2} \right) \mK^{1/2}
\end{equation*}
has no imaginary component, and thus $\mFtrue \mK^{1/2}$ has all real entries.

Turning our attention to Equation~\eqref{eq:FsqrtK:entry}, we again use the fact that
\begin{equation*}
\frac{\mFtrue - \conjugate{\mFtrue}}{2} \mK^{1/2}
= \frac{\imag}{2} \sum_{\ell=2}^{\lfloor (T-1)/2 \rfloor}
    \left( \frac{\mFtrue - \conjugate{\mFtrue}}{2} \right)
			\left(\ve_\ell - \ve_{T-\ell+2}\right)
			\left(\ve_\ell - \ve_{T-\ell+2}\right)^\top
\end{equation*}
to write, for $i \in [n]$ and $k \in \{2,3,\dots,T \}$,
\begin{equation*} \begin{aligned}
\left[ \frac{\mFtrue - \conjugate{\mFtrue}}{2} \mK^{1/2} \right]_{i,k}
&= \frac{\imag}{2} \sum_{\ell=2}^{\lfloor (T-1)/2 \rfloor}
	\sum_{m=1}^T
	\left( \frac{\mFtrue - \conjugate{\mFtrue}}{2} \right)_{i,m}
	\left(\ve_\ell - \ve_{T-\ell+2}\right)_m
	\left(\ve_\ell - \ve_{T-\ell+2}\right)_k \\
&= \frac{\imag}{2} \sum_{\ell=2}^{\lfloor (T-1)/2 \rfloor}
	\left[
	\left( \frac{\mFtrue - \conjugate{\mFtrue}}{2} \right)_{i,\ell}
	\left(\ve_\ell - \ve_{T-\ell+2}\right)_k 
-
\left( \frac{\mFtrue - \conjugate{\mFtrue}}{2} \right)_{i,T-\ell+2}
	\left(\ve_\ell - \ve_{T-\ell+2}\right)_k \right] \\
&= \frac{\imag}{2} \sum_{\ell=2}^{\lfloor (T-1)/2 \rfloor}
	\left( \Ftrue_{i,\ell} - \Ftrue_{i,T-\ell+2} \right)
	\left(\ve_\ell - \ve_{T-\ell+2}\right)_k  ,
\end{aligned} \end{equation*}
where the last equality follows from Equation~\eqref{eq:conjugateSymmetry}.
Noting that the only contributing terms in the sum occur when $k = \ell$ or $k=T-\ell+2$, only one of which is possible, since $\ell \le \lfloor (T-1)/2 \rfloor$, and using Equation~\eqref{eq:conjugateSymmetry} again,
\begin{equation*} 
\left[ \frac{\mFtrue - \conjugate{\mFtrue}}{2} \mK^{1/2} \right]_{i,k}
= \frac{\imag}{2}\left( \Ftrue_{i,k} - \conjugate{\Ftrue_{i,k}} \right)
= -\Im\left( \Ftrue_{i,k} \right) .
\end{equation*}


A parallel argument yields that
\begin{equation*}
\left[ \frac{\mFtrue + \conjugate{\mFtrue}}{2} \mK^{1/2} \right]_{i,k}
= \frac{1}{2}\left( \Ftrue_{i,k} + \conjugate{\Ftrue_{i,k}} \right) ,
\end{equation*}
the real part of $\Ftrue_{i,k}$.
Combining the above two displays yields Equation~\eqref{eq:FsqrtK:entry}.
\end{proof}

\section{Additional Numerical Experiments} \label{apx:experiments}

Here we collect additional numerical experiments in support of our theoretical findings.

\subsection{Effect of Noise Tail Decay} \label{subsec:apx:noise}

Our experiments in the main text focus on the setting where the entries of the noise matrix $\mN$ are generated from a Gaussian.
While our theoretical results require that each row $\mN_i \in \R^T$ be a subgaussian random variable, we suspect that this can be relaxed substantially.
In what follows, we explore the effect of replacing this Gaussian noise with Laplacian noise.
We generate the rows of $\mZtrue$ by drawing i.i.d.~from a standard normal and renormalizing so that $\| \mZtrue \|_F = n$, so that the average row norm of $\mZtrue$ is $1$, but note that the row norms of $\mZtrue$ are not necessarily of the same order.
We then generate $\mZobsd = \mZtrue + \mN$, with the entries of $\mN \in \R^{n \times T}$ drawn i.i.d.~from either a Gaussian or a Laplacian distribution with variance $\nu$.
Note that unlike the experiments in Section~\ref{subsec:expt:variance}, we do not change the variance of the entries of $\mN$ from row to row to fix $\SNR$.
We conducted this experiment with $n=100$ and $T=100,500,1000$ for varying choices of variance $\nu$, with $50$ independent Monte Carlo trials for each combination of conditions.
Figure~\ref{fig:apx:gaussian:tti-by-vary} summarizes the results of this experiment, showing estimation error in $(\tti)$-norm as a function of $\nu$ for the ASE as well as the PCA-based estimate and na\"{i}ve baseline described in Section~\ref{sec:expts}.
Note that the positive direction of the x-axis corresponds to {\em decreasing} values of the SNR-like quantity $\SNR$.
Examining the figure, we note that the performance of the ASE is largely insensitive to whether the observation noise is Gaussian or Laplacian, supporting our conjecture that our results can be extended to subexponential or subgamma noise.
It is interesting to note, as an aside, that the poor performance of the PCA-based method at smaller variances only appears in the case of Gaussian noise.
We suspect that this is due to the fact that our PCA-based method renormalizes the rows of $\mZtrue$ {\em after} performing rank truncation.
In the situation where one or more rows of $\mZobsd$ are much smaller than the others and the noise is small, rank truncation may be sending one or more of these rows to be close to zero, with the result that they are especially bad estimates of the corresponding rows of $\mZtruetilde$ once the rows are renormalized.
The heavier tails of the Laplacian noise prevent this from happening.

\begin{figure}
    \centering
    \includegraphics[width=\textwidth]{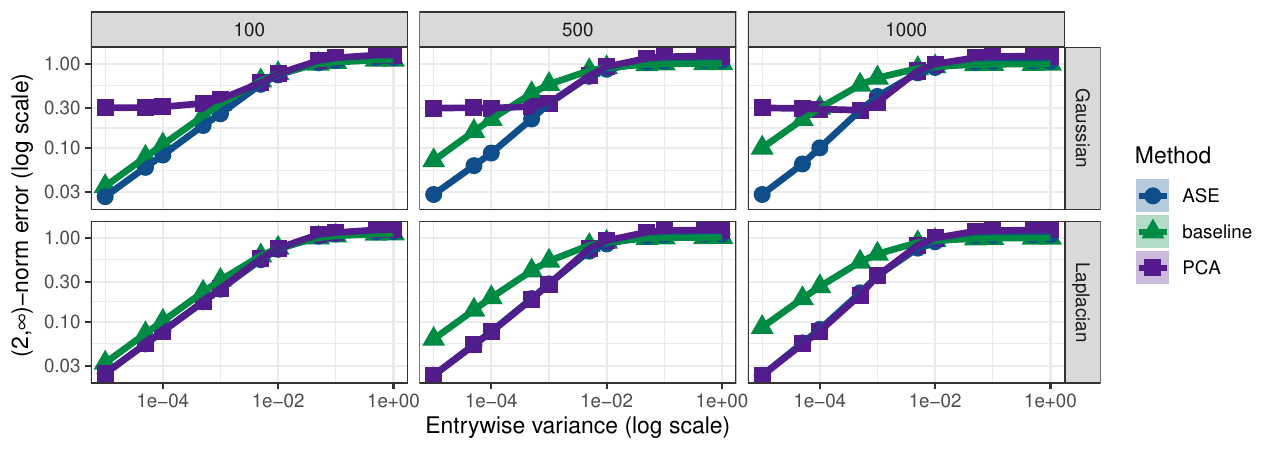}
    \caption{$(\tti)$-norm estimation error as a function of entrywise noise variance under Gaussian (top) and Laplacian (bottom) noise for time series of length $T=100$ (left) $T=500$ (middle) and $T=1000$ (right) and Gaussian (top) and Laplacian (bottom) noise.
Performance of three estimators is shown: the ASE (blue circles), PCA (purple squares) and a na\"ive baseline (green triangles).
Each data point indicates the mean of $50$ Monte Carlo estimates.
Error bars (obscured by the lines) indicate two standard errors of the mean.}
    \label{fig:apx:gaussian:tti-by-vary}
\end{figure}

\subsection{Effect of embedding dimension}

As mentioned in Section~\ref{sec:results}, selecting the embedding dimension is a fundamental task in network embeddings and other dimensionality reduction techniques.
To understand how the choice of embedding dimension $d$ influences estimation accuracy in correlation networks, we generate data as in the experiment in Section~\ref{subsec:apx:noise}, taking the true number of signal frequencies to be $d_0=15$, so that the true embedding dimension is $30$.
We generate $\mZtrue \in \R^{n \times T}$ as in Section~\ref{sec:expts}, and generate the entries of the noise $\mN \in \R^{n \times T}$ i.i.d.~according to a mean zero normal with variance $\nu$.
We then use this to construct the noisy correlation matrix $\mRobsd$, and embed $\mRobsd$ using the ASE with embedding dimension ranging from $1$ to $60$ and recorded the $(\tti)$-norm error in recovering the matrix of true standardized time series $\mZtruetilde \in \R^{n \times T}$.
The results of this experiment, with $n=200, \nu=0.001$ and three choices of time series length $T$, are summarized in Figure~\ref{fig:dimsel}.
Examining the figure, we see that the ASE achieves is best performance when the embedding dimension is chosen correctly, but the ASE outperforms the na\"ive baseline for a wide range of embedding dimensions.
The ASE is better than the baseline once the embedding dimension is at least 25, with this threshold decreasing for larger values of time series length $T$.
These results are broadly in line with the folklore in the network analysis literature (and model selection literature more broadly) that it is better to embed into too many dimensions than too few.
See \cite{TaiLev2026} and citations therein for further discussion of this point.

\begin{figure} 
    \centering
    \includegraphics[width=\textwidth]{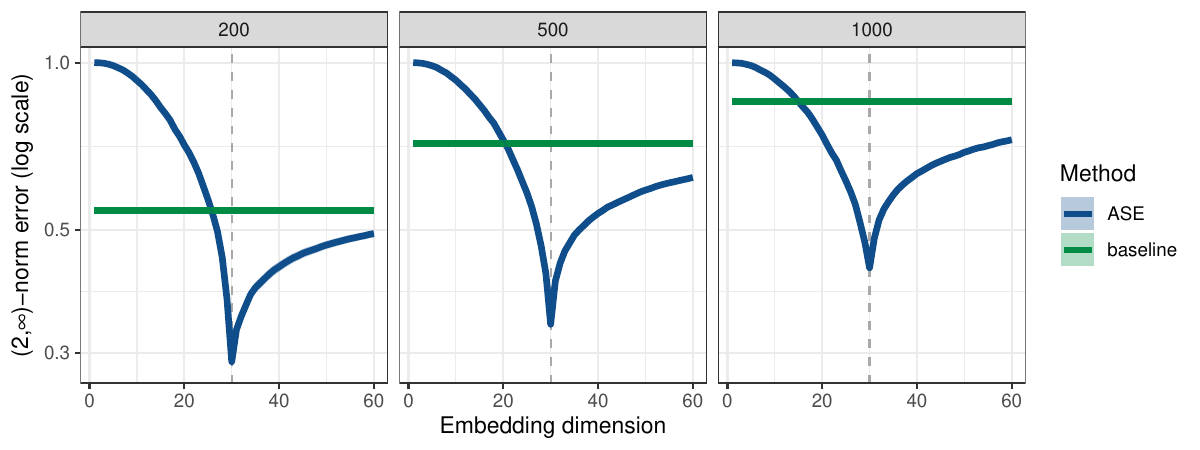}
    \caption{Estimation error in $(\tti)$-norm as a function of embedding dimension $d$ for time series length $T=200,500,1000$ with $d_0 = 15$.
The blue line indicates the performance of the ASE.
The green line indicates the performance of our baseline method, in which $\mZobsdtilde$ is used to estimate $\mZtruetilde$ directly (and thus does not depend on embedding dimension). The vertical grey dashed line indicates the true embedding dimension $30$.}
    \label{fig:dimsel}
\end{figure}

\section{Estimating Signal Variances} \label{apx:sigma}

Here we collect results relating the signal powers $\sigmatrue{i}^2$ as defined in Equation~\eqref{eq:def:sigmatrue} to their empirical estimates $\sigmaobsd{i}^2$, as defined in Equation~\eqref{eq:def:sigmaobsd}.
These technical results will be used in the technical appendices to follow, culminating in the proof of Theorem~\ref{thm:XobsdXtrue:tti} given in Appendix~\ref{apx:thm:tti}.

\begin{lemma} \label{lem:Ni:subgauss}
Suppose that Assumption~\ref{assum:N} holds.
Then with probability at least $1-O(n^{-2})$,
\begin{equation*}
\max_{i \in [n]} \frac{ \| \mN_i \|^2 }{ \nu_i } \le C T \log n.
\end{equation*}
\end{lemma}
\begin{proof}
By assumption, for each $i \in [n]$, $\mN_i \in \R^T$ is a $\nu_i$-subgaussian random vector.
By a standard $\varepsilon$-net argument \citep[see, e.g.,][Chapter 4]{Vershynin2020}, for any $\tau > 0$,
\begin{equation*}
\Pr\left[ \left\| \mN_i \right\| \ge \tau \right]
\le C_0 \exp\left\{ T - \frac{ C_1 \tau^2 }{ \nu_i } \right\} ,
\end{equation*}
for suitably-chosen positive constants $C_0$ and $C_1$.
Taking $\tau = \sqrt{C \nu_i T \log n}$ for $C>0$ suitably large and rearranging,
\begin{equation*}
\Pr\left[ \frac{ \left\| \mN_i \right\|^2 }{ \nu_i } \ge C T \log n \right]
\le C n^{-3} .
\end{equation*}
A union bound over all $i \in [n]$ then yields that with probability at least $1-O(n^{-2})$,
\begin{equation*}
\max_{i \in [n]} \frac{ \left\| \mN_i \right\|^2 }{ \nu_i }
\le C T \log n,
\end{equation*}
as we set out to show.
\end{proof}

\begin{lemma} \label{lem:sigma2:close}
Suppose that Assumption~\ref{assum:N} holds.
Then with high probability, it holds uniformly over all $i \in [n]$ that
\begin{equation*}
\left| \sigmaobsd{i}^2 - \sigmatrue{i}^2 \right|
\le C \nu_i \log n + \sqrt{ \frac{ \sigmatrue{i}^2 \nu_i \log n }{ T } }
\end{equation*} 
\end{lemma}
\begin{proof}
For $i \in [n]$, recalling $\sigmatrue{i}^2$ and $\sigmaobsd{i}^2$ from Equations~\eqref{eq:def:sigmatrue} and~\eqref{eq:def:sigmaobsd}, respectively,
\begin{equation*} 
\sigmaobsd{i}^2 - \sigmatrue{i}^2
= \frac{1}{T} \sum_{t=1}^T N_{i,t}^2
	+ \frac{2}{T} \sum_{t=1}^T N_{i,t} \Ztrue_{i,t}
	- \left( \frac{1}{T} \sum_{t=1}^T N_{i,t} \right)^2
	- \frac{ 2 \mu_i }{ T } \sum_{t=1}^T N_{i,t}.
\end{equation*}
Rearranging and recalling the centering matrix $\mM$ from Equation~\eqref{eq:def:M},
\begin{equation} \label{eq:sigma2conc:expand}
\sigmaobsd{i}^2 - \sigmatrue{i}^2
= \frac{1}{T} \| \mM \mN_i \|^2 
	+ \frac{2}{T} \mN_i^\top \mM \vZtrue{i} .
\end{equation}

By assumption, each $\mN_i$ is a $\nu_i$-subgaussian random vector.
Noting that $\| \mM \vZtrue{i} \|^2 = T \sigmatrue{i}^2$,
standard concentration inequalities \citep{Vershynin2020} yield that with high probability,
\begin{equation*} 
\left| \frac{2}{T} \mN_i^\top \mM \vZtrue{i} \right|
\le \frac{C}{\sqrt{T}} \sqrt{ \nu_i \sigmatrue{i}^2 \log n } 
~\text{ for all }~i \in [n].
\end{equation*}

Taking absolute values in Equation~\eqref{eq:sigma2conc:expand} and applying the above bound, 
\begin{equation} \label{eq:sigmabound:inter}
\left| \sigmaobsd{i}^2 - \sigmatrue{i}^2 \right|
\le 
\frac{1}{T} \| \mM \mN_i \|^2 
+ \frac{C}{\sqrt{T}} \sqrt{ \nu_i \sigmatrue{i}^2 \log n }
~\text{ for all }~i \in [n] .
\end{equation}

Applying submultiplicativity followed by Lemma~\ref{lem:Ni:subgauss},
\begin{equation*}
\max_{i \in [n]} \left\| \mM \mN_i \right\|^2 
\le \max_{i \in [n]} \left\| \mN_i \right\|^2
\le C \max_{i \in [n]} \nu_i T \log n .
\end{equation*}
Applying this bound to Equation~\eqref{eq:sigmabound:inter}, it holds with high probability that
\begin{equation*}
\left| \sigmaobsd{i}^2 - \sigmatrue{i}^2 \right|
\le
C\left( \nu_i \log n + \sqrt{ \frac{ \sigmatrue{i}^2 \nu_i \log n }{ T } } 
	\right) ~\text{ for all }~i \in [n],
\end{equation*}
as we set out to show.
\end{proof}

\begin{lemma} \label{lem:sigmaratio}
Suppose that Assumptions~\ref{assum:N} and~\ref{assum:Rspec} hold.
Then with high probability, it holds uniformly over all $i \in [n]$ that
\begin{equation} \label{eq:sigmaratio:main}
\max \left\{
\left| \frac{ \sigmaobsd{i} }{ \sigmatrue{i} } - 1 \right|,
\left| \frac{ \sigmaobsd{i} }{ \sigmatrue{i} } - 1 \right|
\right\}
\le C\left( \frac{ \log n }{ \SNR }
	+ \sqrt{ \frac{ \log n }{ T \SNR } } \right).
\end{equation}
Further, with high probability, 
\begin{equation} \label{eq:sigmaratio:kicker}
\max_{i \in [n]}
\max\left\{ \left| \frac{ \sigmatrue{i} }{ \sigmaobsd{i} } \right|,
	    \left| \frac{ \sigmaobsd{i} }{ \sigmatrue{i} } \right| \right\}
\le C .
\end{equation}
\end{lemma}
\begin{proof}
By Lemma~\ref{lem:sigma2:close}, it holds uniformly over all $i \in [n]$ that
\begin{equation*}
\sigmaobsd{i}^2 \ge \sigmatrue{i}^2 - C\left( \nu_i \log n
        + \sqrt{ \frac{ \sigmatrue{i}^2 \nu_i \log n }{ T } } \right)
= \sigmatrue{i}^2
\left( 1 - \frac{ C \nu_i \log n }{ \sigmatrue{i}^2 }
	- C\sqrt{  \frac{ \nu_i }{ \sigmatrue{i}^2 }  \frac{ \log n }{T} }
	\right) .
\end{equation*}
Applying our growth assumption in Equation~\eqref{eq:assum:noisepower}, it follows that uniformly over $i \in [n]$,
\begin{equation} \label{eq:sigmaobsd:LB}
\sigmaobsd{i}^2 \ge \sigmatrue{i}^2 \left( 1 - \op{1} \right) .
\end{equation}

For all $i \in [n]$, multiplying through by appropriate quantities, we have
\begin{equation} \label{eq:sigmaratio:closeto1}
\frac{ \sigmaobsd{i} }{ \sigmatrue{i} } - 1
= \frac{ \sigmaobsd{i} - \sigmatrue{i} }{ \sigmatrue{i} }
= \frac{ \sigmaobsd{i}^2 - \sigmatrue{i}^2 }
        { \sigmatrue{i} \left( \sigmatrue{i} + \sigmaobsd{i} \right) } .
\end{equation}
Again using Lemma~\ref{lem:sigma2:close},
\begin{equation*} 
\left| \frac{ \sigmaobsd{i} }{ \sigmatrue{i} } - 1 \right|
\le
\frac{ C }{ \sigmatrue{i} \left( \sigmatrue{i} + \sigmaobsd{i} \right) }
\left(  \nu_i \log n
        + \sqrt{ \frac{ \sigmatrue{i}^2 \nu_i \log n }{ T } } \right).
\end{equation*}
Applying Equation~\eqref{eq:sigmaobsd:LB} yields that
\begin{equation*}
\left| \frac{ \sigmaobsd{i} }{ \sigmatrue{i} } - 1 \right|
\le C\left( \frac{ \nu_i \log n }{ \sigmatrue{i}^2 }
	+ \sqrt{ \frac{ \nu_i \log n }{ T \sigmatrue{i}^2 } } \right).
\end{equation*}
To see the same bound for $\sigmatrue{i}/\sigmaobsd{i}$, note that
\begin{equation*}
\frac{ \sigmatrue{i} }{ \sigmaobsd{i} } - 1
= \frac{1}{ \sigmaobsd{i}/\sigmatrue{i} } -1
= \left( 1 - \frac{ \sigmaobsd{i} }{ \sigmatrue{i} } \right)
	\frac{1}{\sigmaobsd{i}/\sigmatrue{i} }.
\end{equation*}
Applying Equation~\eqref{eq:sigmaratio:closeto1} and using our growth assumption in Equation~\eqref{eq:assum:noisepower} yields Equation~\eqref{eq:sigmaratio:main}.
Equation~\eqref{eq:sigmaratio:kicker} then follows from the growth assumption in Equation~\eqref{eq:assum:noisepower}.
\end{proof}

\begin{lemma} \label{lem:sigmaprod:ratio}
Suppose that Assumptions~\ref{assum:N} and~\ref{assum:Rspec} hold.
Then it holds with high probability that, uniformly over all $i,j \in [n]$,
\begin{equation*}
\left| \frac{\sigmatrue{i} \sigmatrue{j}}{\sigmaobsd{i} \sigmaobsd{j}} - 1 \right|
= \Op{ \frac{ \log n }{ \SNR } + \sqrt{ \frac{ \log n }{ T \SNR } } } .
\end{equation*}
\end{lemma}
\begin{proof}
We begin by noting that
\begin{equation*}
\frac {\sigmatrue{i} \sigmatrue{j}} {\sigmaobsd{i} \sigmaobsd{j}} - 1
=
\frac{ \sigmatrue{i} \sigmatrue{j} - \sigmaobsd{i} \sigmaobsd{j} }
	{ \sigmaobsd{i} \sigmaobsd{j} }
=
\frac{ (\sigmatrue{i} - \sigmaobsd{i})\sigmatrue{j} 
	+ \sigmaobsd{i} (\sigmatrue{j} - \sigmaobsd{j}) }
	{ \sigmaobsd{i} \sigmaobsd{j} }.
\end{equation*}
Multiplying through by appropriate quantities,
\begin{equation*} 
\frac {\sigmatrue{i} \sigmatrue{j}} {\sigmaobsd{i} \sigmaobsd{j}} - 1
=
\frac{ (\sigmatrue{i}^2 - \sigmaobsd{i}^2) \sigmatrue{j} }
	{ \sigmaobsd{i} \sigmaobsd{j} (\sigmatrue{i} + \sigmaobsd{i}) }
+
\frac{ (\sigmaobsd{j}^2 - \sigmatrue{j}^2) }
	{ \sigmaobsd{j} (\sigmatrue{j} + \sigmaobsd{j}) }.
\end{equation*}
Applying the triangle inequality followed by Equation~\eqref{eq:sigmaratio:kicker} from Lemma~\ref{lem:sigmaratio},
it holds uniformly over all $i,j \in [n]$ that
\begin{equation} \label{eq:sigprod:step1}
\left|
\frac {\sigmatrue{i} \sigmatrue{j}} {\sigmaobsd{i} \sigmaobsd{j}} - 1
\right|
\le
\frac{ C\left| \sigmatrue{i}^2 - \sigmaobsd{i}^2 \right| }
	{ \sigmaobsd{i} (\sigmatrue{i} + \sigmaobsd{i}) }
+
\frac{ \left| \sigmaobsd{j}^2 - \sigmatrue{j}^2 \right| }
	{ \sigmaobsd{j} (\sigmatrue{j} + \sigmaobsd{j}) }.
\end{equation}

For ease of notation, for each $i \in [n]$ define
\begin{equation} \label{eq:def:Bi}
B_i = \nu_i \log n +  \sqrt{ \frac{ \sigmatrue{i}^2 \nu_i \log n }{ T } }.
\end{equation}
By Lemma~\ref{lem:sigma2:close}, uniformly over all $i \in [n]$,
\begin{equation*}
\left| \sigmaobsd{i}^2 - \sigmatrue{i}^2 \right| \le C B_i .
\end{equation*}
Thus, applying this bound to Equation~\eqref{eq:sigprod:step1}, it holds uniformly over all $i,j \in [n]$ that
\begin{equation} \label{eq:sigprod:step2}
\left| \frac{\sigmaobsd{i} \sigmaobsd{j}}{\sigmatrue{i} \sigmatrue{j}} - 1 \right|
\le
\frac{ C B_i }{ \sqrt{\sigmatrue{i}^2 - CB_i}
		\left(\sigmatrue{i} + \sqrt{\sigmatrue{i}^2 - CB_i}\right) }
+
\frac{ C B_j }{ \sqrt{\sigmatrue{j}^2-CB_j} 
		\left(\sigmatrue{j} + \sqrt{\sigmatrue{j}^2-CB_j}\right) }.
\end{equation}
By our assumption in Equation~\eqref{eq:assum:noisepower},
\begin{equation*}
\max_{i \in [n]} \frac{ B_i }{ \sigmatrue{i}^2 } = \op{ 1 }.
\end{equation*}
Applying this fact to Equation~\eqref{eq:sigprod:step2}, it holds with high probability that, 
\begin{equation*}
\left| \frac{\sigmaobsd{i} \sigmaobsd{j}}{\sigmatrue{i} \sigmatrue{j}} - 1 \right|
\le
\frac{ C B_i }{ \sigmatrue{i}^2 }
+
\frac{ C B_j }{ \sigmatrue{j}^2 }
~\text{ for all }~i,j \in [n] .
\end{equation*}
Taking a maximum over all $i,j \in [n]$ and recalling the definition of $B_i$ from Equation~\eqref{eq:def:Bi},
\begin{equation*}
\left| \frac{\sigmaobsd{i} \sigmaobsd{j}}{\sigmatrue{i} \sigmatrue{j}} - 1 \right|
\le
C\left[ \max_{i \in [n]} \frac{ \nu_i \log n }{ \sigmatrue{i}^2 }
	+ \max_{i \in [n]} \sqrt{ \frac{ \nu_i \log n }{ T \sigmatrue{i}^2 } }
	\right].
\end{equation*}
Recalling the definition of $\SNR$ from Equation~\eqref{eq:def:SNR} completes the proof.
\end{proof}

\section{Controlling Matrix Norms} \label{apx:norms}

Here we collect technical results related to controlling the variation of $\mRobsd$ about $\mRtrue$.
These are used in Appendix~\ref{apx:subspaces} below, as well as in our proof of Theorem~\ref{thm:XobsdXtrue:tti} given in Appendix~\ref{apx:thm:tti}.

\begin{lemma} \label{lem:Rtrue:facts}
With $\mRtrue$ as in Equation~\eqref{eq:def:Rtrue}, $\sqrt{n} \le \|\mRtrue\|_F \le n$.
Further, the leading eigenvalue of $\mRtrue$ obeys $1 \le \lambdatrue{1} \le n$.
\end{lemma}
\begin{proof}
The lower-bound on $\|\mRtrue\|_F$ follows from the fact that all diagonal entries of $\mRtrue$ are $1$ by construction.
The upper bound follows from the fact that all entries of $\mRtrue$ are bounded between $-1$ and $1$.
The lower-bound on $\lambdatrue{1}$ follows from the fact that for any standard basis vector $\ve_i$, $i \in [n]$, we have $\ve_i^\top \mRtrue \ve_i = 1$ (again because $\mRtrue$ has all diagonal entries equal to one).
The upper bound follows from the fact that 
$\lambdatrue{1} \le \| \mRtrue \|_F \le n$.
\end{proof}

\begin{lemma} \label{lem:N:max}
Suppose that Assumptions~\ref{assum:N} and~\ref{assum:Rspec} hold.
Then with high probability, it holds uniformly over $i \in [n]$ that
\begin{equation} \label{eq:bound:Nrow}
\| \mN_i \|^2 \le C \nu_i T \log T .
\end{equation}
Further, 
\begin{equation} \label{eq:SigmatrueN:Lagnap}
\left\| \mSigmatrue^{-1/2} \mN \right\|_{\tti}
= \Op{ \sqrt{ \frac{ T \log T }{ \SNR } } } .
\end{equation}
\end{lemma}
\begin{proof}
Note that for a fixed basis vector $\ve_t \in \R^T$, since $\mN_i$ is a subgaussian random vector by assumption, it holds with high probability that
\begin{equation*}
\left| \ve_t^\top \mN_i \right|
\le \sqrt{ C \nu_i \log T}.
\end{equation*}
Choosing the constant $C>0$ large enough, we can ensure that this bound holds
uniformly over all $t \in [T]$, so that
\begin{equation*}
\| \mN_i \|^2
= \sum_{t=1}^T \left( \ve_t^\top \mN_i \right)^2
\le C T \nu_i \log T .
\end{equation*}
Again choosing $C>0$ suitably large and recalling our assumption in Equation~\eqref{eq:assum:nT}, a union bound ensures that this holds
for all $i \in [n]$ with high probability, establishing Equation~\eqref{eq:bound:Nrow}.
Noting that
\begin{equation*}
\left\| \mSigmatrue^{-1/2} \mN \right\|_{\tti}
= \max_{i \in [n]} \frac{ \| \mN_i \| }{ \sigmatrue{i} },
\end{equation*}
Equation~\eqref{eq:SigmatrueN:Lagnap} follows immediately, once we recall the definition of $\SNR$ from Equation~\eqref{eq:def:SNR}.
\end{proof}

\begin{lemma} \label{lem:SigmatrueN:spectral}
Suppose that Assumption~\ref{assum:N} holds.
\begin{equation*}
\left\| \mSigmatrue^{-1/2} \mN \right\|
= \Op{ \sqrt{T \frac{\log n }{\SNR} } } .
\end{equation*}
\end{lemma}
\begin{proof}
Let $\vu \in \R^n$ and $\vv \in \R^T$ be fixed unit-norm vectors.
\begin{equation*}
\vu^\top \mSigmatrue^{-1/2} \mN \vv
= \sum_{i=1}^n \frac{ u_i \mN_i^\top \vv }{ \sigmatrue{i} } .
\end{equation*}
Since the rows of $\mN$ are independent subgaussian vectors by Assumption~\ref{assum:N}, standard concentration inequalities \citep{Vershynin2020} yield that with high probability,
\begin{equation*}
\left| \sum_{i=1}^n \frac{ u_i \mN_i^\top \vv }{ \sigmatrue{i} } \right|
\le C \sqrt{ \sum_{i=1}^n u_i^2 \frac{\nu_i}{\sigmatrue{i}^2} v_i^2 \log n}
\le C \sqrt{ \frac{\log n }{\SNR} } .
\end{equation*}
Modifying the concentration inequality to encompass a union bound over $\varepsilon$-nets covering the unit spheres in $\R^n$ and $\R^T$ \citep{Vershynin2020}, it holds with high probability that
\begin{equation*}
\left\| \mSigmatrue^{-1/2} \mN \right\|
= \Op{ \sqrt{ \frac{ (n+T) \log n}{\SNR} } } .
\end{equation*}
Using our assumption that $n = O(T)$ completes the proof.
\end{proof}

\begin{lemma} \label{lem:SigmaobsdN:subgauss}
Suppose that Assumption~\ref{assum:N} holds.
Then
\begin{equation*}
\max\left\{ \left\| \mSigmaobsd^{-1/2} \mN \right\|,
                \left\| \mSigmatrue^{-1/2} \mN \right\| \right\}
= \Op{ \sqrt{ \frac{ T \log n}{\SNR} } } 
\end{equation*}
\end{lemma}
\begin{proof}
Multiplying through by appropriate quantities and using submultiplicativity,
\begin{equation*}
\left\| \mSigmaobsd^{-1/2} \mN \right\|
\le \left\| \mSigmaobsd^{-1/2} \mSigmatrue^{1/2} \right\|
	\left\|  \mSigmatrue^{-1/2} \mN \right\| .
\end{equation*}
Applying Lemma~\ref{lem:sigmaratio},
\begin{equation*}
\left\| \mSigmaobsd^{-1/2} \mN \right\|
\le C \left\| \mSigmatrue^{-1/2} \mN \right\|,
\end{equation*}
and Lemma~\ref{lem:SigmatrueN:spectral} completes the proof.
\end{proof}

\begin{lemma} \label{lem:RobsdRtrue:frob}
Suppose that Assumption~\ref{assum:N} holds.
Then
\begin{equation*}
\left\| \mRobsd - \mRtrue \right\|_F = \op{ \| \mRtrue \|_F }.
\end{equation*}
\end{lemma}
\begin{proof}
For $i,j \in [n]$, we have
\begin{equation*} \begin{aligned}
\left( \mRobsd - \mRtrue \right)_{i,j}
&= \frac{ \mZobsd_i^\top \mM \mZobsd_j }{ T \sigmaobsd{i} \sigmaobsd{j} }
- \frac{ \vZtrue{i}^\top \mM \vZtrue{j} }{ T \sigmatrue{i} \sigmatrue{j} } \\
&= 
\frac{1}{T \sigmaobsd{i} \sigmaobsd{j} } 
	\left( \mZobsd_i^\top \mM \mZobsd_j 
				- \vZtrue{i}^\top \mM \vZtrue{j} \right) 
+ \frac{ \vZtrue{i}^\top \mM \vZtrue{j} }{ T \sigmatrue{i} \sigmatrue{j} }
	\left( \frac{ \sigmatrue{i} \sigmatrue{j} }{ \sigmaobsd{i} \sigmaobsd{j} } 
			- 1 \right) .
\end{aligned} \end{equation*}
Squaring and summing over all $i,j\in [n]$, noting that $\mRobsd$ and $\mRtrue$
have all diagonal entries equal to $1$ by construction,
\begin{equation*} \begin{aligned}
\left\| \mRobsd - \mRtrue \right\|_F^2
&= \frac{1}{T^2} \sum_{i=1}^n \sum_{j : j \neq i}
\left( \frac{\sigmatrue{i} \sigmatrue{j}}
		{\sigmaobsd{i} \sigmaobsd{j}} \right)^2
	\left( \frac{ \mZobsd_i^\top \mM \mZobsd_j 
					- \vZtrue{i}^\top \mM \vZtrue{j} }
			{ \sigmatrue{i} \sigmatrue{j} } \right)^2
+
\sum_{i=1}^n \sum_{j:j\neq i}
	\left( \frac{ \sigmatrue{i} \sigmatrue{j} }{ \sigmaobsd{i} \sigmaobsd{j} } 
			- 1 \right)^2 {\Rtrue_{i,j}}^2 .
\end{aligned} \end{equation*}
Applying Lemma~\ref{lem:sigmaprod:ratio},
\begin{equation} \label{eq:RobsdRtrue:frob:checkpt} \begin{aligned}
\left\| \mRobsd - \mRtrue \right\|_F^2
&\le
\frac{C}{T^2} \sum_{i=1}^n \sum_{j: j\neq i}
	\left( \frac{ \mZobsd_i^\top \mM \mZobsd_j 
					- \vZtrue{i}^\top \mM \vZtrue{j} }
			{ \sigmatrue{i} \sigmatrue{j} } \right)^2 
+
\left(  \frac{ \log n }{ \SNR } + \sqrt{ \frac{ \log n }{ T \SNR } }
    \right)^2 \left\| \mRtrue \right\|_F^2 .
\end{aligned} \end{equation}

Recalling that $\mZobsd_i = \vZtrue{i} + \mN_i$, we have
\begin{equation*}
\left| \frac{ \mZobsd_i^\top \mM \mZobsd_j - \vZtrue{i}^\top \mM \vZtrue{j} }
			{ \sqrt{T} \sigmatrue{i} \sigmatrue{j} } \right|
\le
\frac{ \left| \mN_i^\top \mM \vZtrue{j} \right|
		+ \left| \vZtrue{i}^\top \mM \mN_j \right|
		+ \left| \mN_i^\top \mM \mN_j \right| }
			{ \sqrt{T} \sigmatrue{i} \sigmatrue{j} } ,
\end{equation*}
and standard subgaussian concentration inequalities and Lemma~\ref{lem:N:max} yield that with high probability, uniformly over all $i,j \in [n]$ distinct,
\begin{equation*}
\left| \frac{ \mZobsd_i^\top \mM \mZobsd_j - \vZtrue{i}^\top \mM \vZtrue{j} }
			{ \sqrt{T} \sigmatrue{i} \sigmatrue{j} } \right|
\le
C \sqrt{ \frac{ \log n }{ \SNR } }
+ \frac{ C \sqrt{ n \nu_i \nu_j } \log T }{ \sqrt{T} \sigmatrue{i} \sigmatrue{j} }
\le 
C \sqrt{ \frac{ \log n }{ \SNR } }
+ C \sqrt{ \frac{n}{T} } \frac{ \log T }{ \SNR } .
\end{equation*}
Applying this to Equation~\eqref{eq:RobsdRtrue:frob:checkpt},
\begin{equation*} \begin{aligned}
\left\| \mRobsd - \mRtrue \right\|_F^2
&\le
\frac{C}{T} \sum_{i=1}^n \sum_{j: j\neq i}
	\left( \sqrt{ \frac{ \log n }{ \SNR } }
		+ \sqrt{ \frac{n}{T} } \frac{ \log T }{ \SNR } \right)^2
+
\left(  \frac{ \log n }{ \SNR } + \sqrt{ \frac{ \log T }{ T \SNR } }
    \right)^2 \left\| \mRtrue \right\|_F^2 \\
&\le \frac{ C n^2 }{T} \left( \sqrt{ \frac{ \log n }{ \SNR } }
        + \sqrt{ \frac{n}{T} } \frac{ \log T }{ \SNR } \right)^2
+
\left(  \frac{ \log n }{ \SNR } + \sqrt{ \frac{ \log T }{ T \SNR } }
    \right)^2 \left\| \mRtrue \right\|_F^2 .
\end{aligned} \end{equation*}
Applying the growth assumptions in Equations~\eqref{eq:assum:nT} and~ \eqref{eq:assum:noisepower}, 
\begin{equation*} \begin{aligned}
\left\| \mRobsd - \mRtrue \right\|_F^2
= \op{ n + \| \mRtrue \|_F^2 }.
\end{aligned} \end{equation*}
Considering the diagonal entries of $\mRtrue$ yields the trivial lower bound $\| \mRtrue \|_F^2 \ge n$, which completes the proof.
\end{proof}

\begin{lemma} \label{lem:RobsdRtrue:tti} 
Suppose Assumption~\ref{assum:N} holds.
Then
\begin{equation*}
\left\| \mRobsd - \mRtrue \right\|_{\tti} = \op{ \| \mRtrue \|_{\tti} }.
\end{equation*}
\end{lemma}
\begin{proof}
The proof follows the same argument as Lemma~\ref{lem:RobsdRtrue:frob}, summing along a row of $\mRobsd - \mRtrue$ rather than summing over all entries.
\end{proof}

\begin{lemma} \label{lem:RobsdRtrue:spectral}
Suppose that Assumption~\ref{assum:N} holds.
Then
\begin{equation*}
\left\| \mRobsd - \mRtrue \right\|
\le
C \left( \frac{ \lambdatrue{1} \log n }{ \SNR } 
        + \sqrt{ \frac{\lambdatrue{1} \log n}{\SNR} } \right) .
\end{equation*}
\end{lemma}
\begin{proof}
Expanding the definitions of $\mRobsd$ and $\mRtrue$ from Equations~\eqref{eq:def:Robsd} and~\eqref{eq:def:Rtrue}, respectively, and applying the triangle inequality,
\begin{equation} \label{eq:RobsdRtrue:spectral:bigtri}
\begin{aligned} 
\left\| \mRobsd - \mRtrue \right\| 
&\le \left\| \frac{1}{T} \mSigmaobsd^{-1/2} \left( \mN \mM \mZtrue^\top + \mZtrue \mM \mN^\top 
                                        + \mN \mM \mN^\top \right)
                \mSigmaobsd^{-1/2} \right\| \\
&~~~~~~+ \left\| \frac{ \mSigmaobsd^{-1/2} \mZtrue \mM \mZtrue^\top \mSigmaobsd^{-1/2}
                        - \mSigmatrue^{-1/2} \mZtrue \mM \mZtrue^\top \mSigmatrue^{-1/2} }{ T }
                \right\|.
\end{aligned} \end{equation}

Adding and subtracting appropriate quantities and applying the triangle inequality,
\begin{equation*} \begin{aligned}
&\left\| \frac{ \mSigmaobsd^{-1/2} \mZtrue \mM \mZtrue^\top \mSigmaobsd^{-1/2}
         \!-\! \mSigmatrue^{-1/2} \mZtrue \mM \mZtrue^\top \mSigmatrue^{-1/2} }
        {T} \right\| \\
&~~~~~~\le 
\left\| \frac{1}{T} \left( \mSigmaobsd^{-1/2} - \mSigmatrue^{-1/2} \! 
        \right) \!  \mZtrue \mM \mZtrue^\top \mSigmaobsd^{-1/2} \right\| 
+ \left\| \frac{1}{T} \mSigmaobsd^{-1/2} \mZtrue \mM \mZtrue^\top
                \left( \mSigmaobsd^{-1/2} - \mSigmatrue^{-1/2} \right) \right\|.
\end{aligned} \end{equation*}
Factoring out $\mSigmatrue^{-1/2}$, recalling the definition of $\mRtrue$ from Equation~\eqref{eq:def:Rtrue} and using submultiplicativity of the norm,
\begin{equation*} \begin{aligned}
\left\| \frac{1}{T} \!
        \left( \mSigmaobsd^{-1/2} \mZtrue \mM \mZtrue^\top \mSigmaobsd^{-1/2}
         \!-\! \mSigmatrue^{-1/2} \mZtrue \mM \mZtrue^\top \mSigmatrue^{-1/2}
                \right) \right\|
&\le \left\| \mSigmaobsd^{-1/2} \mSigmatrue^{1/2} - \mI \right\|
        \left\| \mRtrue \right\|
        \left\| \mSigmatrue^{1/2} \mSigmaobsd^{-1/2} \right\| \\
&~~~~~~+ \left\| \mRtrue \right\| 
		\left\| \mSigmaobsd^{-1/2} \mSigmatrue^{1/2} - \mI \right\|.
\end{aligned} \end{equation*}
Applying Lemma~\ref{lem:sigmaratio}, it holds with high probability that
\begin{equation} \label{eq:RobsdRtrue:spectral:tripart2}
\left\| \frac{ \mSigmaobsd^{-1/2} ZMZ^\top \mSigmaobsd^{-1/2}
        \!-\! \mSigmaobsd^{-1/2} ZMZ^\top \mSigmaobsd^{-1/2} } {T} \right\|
\le C \lambdatrue{1} \left( \frac{ \log n }{ \SNR }
                + \sqrt{ \frac{ \log n }{ T \SNR } } \right). 
\end{equation}

Applying the triangle inequality and basic properties of the norm,
\begin{equation} \label{eq:RhatR:spectral:tripart1:split}
\begin{aligned}
\left\| \frac{1}{T}
                \mSigmaobsd^{-1/2} \!
                \left(\! \mN \mM \mZtrue^\top 
                        \!+\! \mZtrue \mM \mN^\top 
                        \!+\! \mN \mM \mN^\top \! \right) \!
                \mSigmaobsd^{-1/2} \right\|
&\le \frac{ 2 }{T } \left\| \mSigmaobsd^{-1/2} \mN \right\|
                        \left\| \mSigmaobsd^{-1/2} \mZtrue \mM \right\| \\
&~~~~~~+ \frac{1}{T} 
        \left\| \mSigmaobsd^{-1/2} \mN \right\|^2~\left\| \mM \right\| .
\end{aligned} \end{equation}
Using the fact that $\mM$ is a projection and applying Lemma~\ref{lem:SigmaobsdN:subgauss},
\begin{equation*}
\frac{1}{T} \left\| \mSigmaobsd^{-1/2} \mN \right\|^2~\left\| \mM \right\|
\le \frac{C \log n}{\SNR} .
\end{equation*}
Again using Lemma~\ref{lem:SigmaobsdN:subgauss},
\begin{equation*}
\frac{ 1 }{T } \left\| \mSigmaobsd^{-1/2} \mN \right\|
                        \left\| \mSigmaobsd^{-1/2} \mZtrue \mM \right\|
\le
\frac{ C \log n }{ \sqrt{T \SNR} } 
	\left\| \mSigmatrue^{-1/2} \mZtrue \mM \right\|
\le C \sqrt{ \frac{ \lambdatrue{1} \log n }{ \SNR } } ,
\end{equation*}
where the second inequality follows from the definition of $\mRtrue$.
Applying the above two displays to Equation~\eqref{eq:RhatR:spectral:tripart1:split},
\begin{equation*}
\left\| \frac{1}{T}
                \mSigmaobsd^{-1/2} \!
                \left(\! \mN \mM \mZtrue^\top \!
                        \!+\! \mZtrue \mM \mN^\top  \!
                        \!+\! \mN \mM \mN^\top \! \right) \!
                \mSigmaobsd^{-1/2} \right\|
\!\le C\sqrt{ \frac{ \log n }{ \SNR } } \!\!
	\left( \! \sqrt{\lambdatrue{1}} + \!\sqrt{ \frac{ \log n }{ \SNR } } \right)
\le \! C \!\sqrt{ \frac{ \lambdatrue{1} \log n }{ \SNR } } ,
\end{equation*}
where the second inequality follows from the assumptions in Equations~\eqref{eq:assum:noisepower} and~\eqref{eq:assum:lambdad:LB}, along with the trivial bound $\lambdatrue{d} \le \lambdatrue{1}$.
Applying this and Equation~\eqref{eq:RobsdRtrue:spectral:tripart2} to Equation~\eqref{eq:RobsdRtrue:spectral:bigtri},
\begin{equation*}
\left\| \mRobsd - \mRtrue \right\|
\le
C \lambdatrue{1} \left( \frac{ \log n }{ \SNR } 
	+ \sqrt{ \frac{ \log n }{ T \SNR } } \right) 
+
C \sqrt{ \frac{ \lambdatrue{1} \log n }{ \SNR } } .
\end{equation*}
Collecting terms, using Lemma~\ref{lem:Rtrue:facts} to bound $1 \le \lambdatrue{1} \le n$, and applying the assumption in Equation~\eqref{eq:assum:nT},
\begin{equation*}
\left\| \mRobsd - \mRtrue \right\|
\le
C \left( \frac{ \lambdatrue{1} \log n }{ \SNR }
        + \sqrt{ \frac{ \lambdatrue{1} \log n}{\SNR} } \right),
\end{equation*}
completing the proof.
\end{proof}

\begin{lemma} \label{lem:tti:mx:highprob}
Suppose that Assumption~\ref{assum:N} holds and let $\mB \in \R^{T \times p}$ be such that the rows of $\mN$ are independent given $\mB$.
Then
\begin{equation*}
\left\| \mSigmatrue^{-1/2} \mN \mB \right\|_{\tti} 
= \Op{ \| \mB \|_F \sqrt{ \frac{ \log (n+p)}{ \SNR } } } .
\end{equation*}
\end{lemma}
\begin{proof}
By definition,
\begin{equation} \label{eq:SNM:tti:expand}
\left\| \mSigmatrue^{-1/2} \mN \mB \right\|_{\tti}^2
= \max_{i \in [n]} \frac{1}{\sigmatrue{i}^2} 
	\sum_{k=1}^p \left| \mN_{i,\cdot}^\top \mB_{\cdot,k} \right|^2 .
\end{equation}
By standard concentration inequalities, it holds with high probability that, for all $i \in [n]$ and $k \in [p]$,
\begin{equation*}
\left| \mN_{i,\cdot}^\top \mB_{\cdot,k} \right|
\le C \sqrt{ \nu_i \log (n+p) } .
\end{equation*}
It follows that
\begin{equation*}
\left\| \mSigmatrue^{-1/2} \mN \mB \right\|_{\tti}^2
\le C \max_{i \in [n]} \frac{ \nu_i \log (n+p)}{\sigmatrue{i}^2} \sum_{k=1}^n \left\| \mB_{\cdot,k} \right\|^2.
\end{equation*}
Applying this to Equation~\eqref{eq:SNM:tti:expand}, taking square roots and recalling the definition of $\SNR$ from Equation~\eqref{eq:def:SNR} completes the proof.
\end{proof}

\begin{lemma} \label{lem:UEU:frobenius}
Suppose that Assumption~\ref{assum:N} holds.
Then
\begin{equation*}
\left\| \mUtrue^\top \left( \mRobsd - \mRtrue \right) \mUtrue \right\|_F
\le 
C d \left\| \mRtrue \right\|_F
\left( \frac{ \log T }{ \SNR } + \sqrt{ \frac{ \log T }{ T \SNR } } \right) .
\end{equation*}
\end{lemma}
\begin{proof}
Recalling the definitions of $\mRobsd$ and $\mRtrue$ from Equations~\eqref{eq:def:Robsd} and~\eqref{eq:def:Rtrue} and applying the triangle inequality,
\begin{equation} \label{eq:EUE:frob:maintri} \begin{aligned}
\left\| \mUtrue^\top \left( \mRobsd - \mRtrue \right) \mUtrue \right\|_F
&\le 
\frac{1}{T}
\left\| \mUtrue^\top 
	\left( \mSigmaobsd^{-1/2} \mZobsd \mM \mZobsd^\top \mSigmaobsd^{-1/2}
	-  \mSigmatrue^{-1/2} \mZobsd \mM \mZobsd^\top \mSigmatrue^{-1/2}
	\right) \mUtrue \right\|_F \\
&~~~~~~+
\frac{1}{T}
\left\| \mUtrue^\top \mSigmatrue^{-1/2}
	\left( \mZobsd \mM \mZobsd^\top - \mZtrue \mM \mZtrue^\top \right)
	\mSigmatrue^{-1/2} \mUtrue \right\|_F .
\end{aligned} \end{equation}

For $k,\ell \in [d]$,
\begin{equation*} \begin{aligned}
\frac{1}{T} &
\left[ \mUtrue^\top
		\left( \mSigmaobsd^{-1/2} \mZobsd \mM \mZobsd^\top \mSigmaobsd^{-1/2}
        -  \mSigmatrue^{-1/2} \mZobsd \mM \mZobsd^\top \mSigmatrue^{-1/2}
        \right) \mUtrue \right]_{k,\ell} \\
&= \sum_{i=1}^n \sum_{j=1}^n \mRobsd_{i,j}
	\left( \frac{\sigmaobsd{i} \sigmaobsd{j}}{\sigmatrue{i} \sigmatrue{j}}
		- 1 \right) \Utrue{i,k} \mUtrue{j,\ell} ,
\end{aligned} \end{equation*}
and the Cauchy-Schwarz inequality implies
\begin{equation*} \begin{aligned}
& \left[ \frac{1}{T} 
	\mUtrue^\top
	\left( \mSigmaobsd^{-1/2} \mZobsd \mM \mZobsd^\top \mSigmaobsd^{-1/2}
        -  \mSigmatrue^{-1/2} \mZobsd \mM \mZobsd^\top \mSigmatrue^{-1/2}
        \right) \mUtrue \right]_{k,\ell}^2 \\
&~~~\le
\left\| \mRobsd \right\|_F^2
\sum_{i=1}^n \sum_{j=1}^n
	\left( \frac {\sigmaobsd{i} \sigmaobsd{j}} {\sigmatrue{i} \sigmatrue{j}} 
			- 1 \right)^2 {\Utrue{i,k}}^2 {\Utrue{j,\ell}}^2
\le C \left\| \mRobsd \right\|_F^2
	\max_{i,j \in [n]} 
	\left( \frac {\sigmaobsd{i} \sigmaobsd{j}} {\sigmatrue{i} \sigmatrue{j}}
                - 1 \right)^2 ,
\end{aligned} \end{equation*}
where the second inequality follows from the fact that $\mUtrue$ has orthonormal columns.
Applying Lemma~\ref{lem:sigmaprod:ratio}, summing over $k,\ell \in [d]$ and taking square roots,
\begin{equation*} \begin{aligned}
\frac{1}{T} &
\left\| \mUtrue^\top 
	\left( \mSigmaobsd^{-1/2} \mZobsd \mM \mZobsd^\top \mSigmaobsd^{-1/2}
	-  \mSigmatrue^{-1/2} \mZobsd \mM \mZobsd^\top \mSigmatrue^{-1/2}
	\right) \mUtrue \right\|_F \\
&\le C d \left\| \mRobsd \right\|_F 
	\left( \frac{ \log n }{ \SNR } + \sqrt{ \frac{ \log n }{ T \SNR } }
	\right) .
\end{aligned} \end{equation*}
Applying Lemma~\ref{lem:RobsdRtrue:frob},
\begin{equation} \label{eq:USigmaTermsU:frob} \begin{aligned}
\frac{1}{T} &
\left\| \mUtrue^\top 
	\left( \mSigmaobsd^{-1/2} \mZobsd \mM \mZobsd^\top \mSigmaobsd^{-1/2}
	-  \mSigmatrue^{-1/2} \mZobsd \mM \mZobsd^\top \mSigmatrue^{-1/2}
	\right) \mUtrue \right\|_F \\
&\le C d \left\| \mRtrue \right\|_F 
	\left( \frac{ \log n }{ \SNR } + \sqrt{ \frac{ \log n }{ T \SNR } } \right) .
\end{aligned} \end{equation}

Again fixing $k,\ell \in [d]$ and expanding the matrix-vector products,
\begin{equation*} \begin{aligned}
\frac{1}{T} \left[ \mUtrue^\top \mSigmatrue^{-1/2}
                \mN \mM \mZtrue^\top 
                \mSigmatrue^{-1/2} \mUtrue \right]_{k,\ell}
&= \frac{1}{T} \sum_{i=1}^n \sum_{j=1}^n
	\frac{ \mN_i^\top \mM \vZtrue{j} }{ \sigmatrue{i} \sigmatrue{j} }
	\Utruevec{ik} \Utruevec{j\ell}  \\
&= \frac{1}{\sqrt{T}}
	\sum_{i=1}^n \frac{ \Utruevec{ik} \mN_i^\top }{ \sigmatrue{i} }
	\sum_{j=1}^n \frac{ \Utruevec{j\ell} \mM \vZtrue{j} }
			{ \sqrt{T} \sigmatrue{j} } .
\end{aligned} \end{equation*}
Since the rows of $\mN$ are independent by assumption, standard concentration inequalities yield
\begin{equation*} \begin{aligned}
\left| \frac{1}{T} \left[ \mUtrue^\top \mSigmatrue^{-1/2}
                \mN \mM \mZtrue^\top 
                \mSigmatrue^{-1/2} \mUtrue \right]_{k,\ell} \right|
&\le \frac{1}{\sqrt{T}}
\sqrt{ \sum_{i=1}^n \frac{ \Utruevec{ik}^2 \nu_i \log n }{ \sigmatrue{i}^2 } }
	\left\| \sum_{j=1}^n \frac{ \mM \vZtrue{j} \Utruevec{j\ell} }
		{ \sqrt{T} \sigmatrue{j} } \right\| \\
&\le \frac{1}{\sqrt{T}} \sqrt{ \frac{ \log n }{ \SNR } }
	\sqrt{ \sum_{j=1}^n \left\| \frac{ \mM \vZtrue{j} }
		{ \sqrt{T} \sigmatrue{j} } \right\|^2 } ,
\end{aligned} \end{equation*}
where the second inequality follows from the triangle and Cauchy-Schwarz inequalities.
Since $\| \mM \vZtrue{j} / \sqrt{T} \sigmatrue{j} \| = 1$ for all $j \in [n]$,
\begin{equation*}
\left| \frac{1}{T} \left[ \mUtrue^\top \mSigmatrue^{-1/2}
                \mN \mM \mZtrue^\top 
                \mSigmatrue^{-1/2} \mUtrue \right]_{k,\ell} \right|
\le \sqrt{ \frac{n}{T}  \frac{ \log n }{ \SNR } } .
\end{equation*}
Squaring, summing over $k,\ell \in [d]$, and taking square roots,
\begin{equation} \label{eq:UEU:NZbound} \begin{aligned}
\left\|  \frac{1}{T} \mUtrue^\top \mSigmatrue^{-1/2}
                \mN \mM \mZtrue^\top 
                \mSigmatrue^{-1/2} \mUtrue \right\|_F
\le C d \sqrt{ \frac{ n \log n }{ T \SNR } } .
\end{aligned} \end{equation}

Once more fixing $k,\ell \in [d]$ and expanding,
\begin{equation*} 
\left[ \frac{1}{T} \mUtrue^\top \mSigmatrue^{-1/2}
                \mN \mM \mN^\top 
                \mSigmatrue^{-1/2} \mUtrue \right]_{k,\ell}
= \sum_{i=1}^n 
	\frac{ \mN_i^\top \mM \mN_i \Utrue{i,k} \Utrue{i,\ell} }
		{ T \sigmatrue{i}^2 } 
+ \sum_{i=1}^n \sum_{j : j \neq i}
	\frac{ \mN_i^\top \mM \mN_j \Utrue{i,k} \Utrue{j,\ell} }
		{ T \sigmatrue{i} \sigmatrue{j} } .
\end{equation*}
Applying Lemma~\ref{lem:N:max},
\begin{equation*}
\left| \left[ \frac{1}{T} \mUtrue^\top \mSigmatrue^{-1/2}
                \mN \mM \mN^\top 
                \mSigmatrue^{-1/2} \mUtrue \right]_{k,\ell} \right|
\le
\frac{C \log T}{\SNR} \sum_{i=1}^n |\Utrue{i,k} \Utrue{i,\ell}|
+ \left| \sum_{i=1}^n \sum_{j : j \neq i}
	\frac{ \mN_i^\top \mM \mN_j \Utrue{i,k} \Utrue{j,\ell} }
		{ T \sigmatrue{i} \sigmatrue{j} } \right| .
\end{equation*}
Applying the Cauchy-Schwarz and triangle inequalities,
\begin{equation} \label{eq:USNMNSU:finalstep} 
\left| \left[ \frac{1}{T} \mUtrue^\top \mSigmatrue^{-1/2}
                \mN \mM \mN^\top 
                \mSigmatrue^{-1/2} \mUtrue \right]_{k,\ell} \right|
\le
\frac{C \log T}{\SNR} 
+ \sum_{i=1}^n \left| \frac{ \Utrue{i,k} \mN_i^\top \mM }{ \sigmatrue{i} }
	\sum_{j : j \neq i} \frac{ \mN_j \Utrue{j,\ell} }
							{ T \sigmatrue{j} } \right| ,
\end{equation}
Since the rows of $\mN$ are independent by assumption,
subgaussian concentration inequalities yield that with high probability,
for all $i \in [n]$,
\begin{equation*}
\left| \frac{ \Utrue{i,k} \mN_i^\top \mM }{ \sigmatrue{i} }
	\sum_{j : j \neq i}
	\frac{ \mN_j \Utrue{j,\ell} }
		{ T \sigmatrue{j} } \right|
\le
\frac{ C }{T}
\left\| \frac{ \Utrue{i,k} \mN_i^\top \mM }{ \sigmatrue{i} } \right\|
\sqrt{ \sum_{j: j \neq i} \frac{ {\Utrue{j,\ell}}^2 \nu_j \log T }
							{ \sigmatrue{j}^2 } }
\le
\frac{ C }{T}
\left\| \frac{ \Utrue{i,k} \mN_i^\top \mM }{ \sigmatrue{i} } \right\|
\sqrt{ \frac{ \log T }{ \SNR } } .
\end{equation*}
Applying this to Equation~\eqref{eq:USNMNSU:finalstep} and using
the Cauchy-Schwarz inequality along with the fact that $\mM$ is a projection,
\begin{equation*} 
\left| \left[ \frac{1}{T} \mUtrue^\top \mSigmatrue^{-1/2}
                \mN \mM \mN^\top 
                \mSigmatrue^{-1/2} \mUtrue \right]_{k,\ell} \right|
\le
\frac{C \log T}{\SNR} 
+ \frac{ C }{T} \sqrt{ \frac{\log T}{\SNR} }
	\sqrt{ \sum_{i=1}^n \frac{ \| \mN_i \|^2 }{ \sigmatrue{i}^2 } } .
\end{equation*}
Again using Lemma~\ref{lem:N:max},
\begin{equation*}
\left| \left[ \frac{1}{T} \mUtrue^\top \mSigmatrue^{-1/2}
                \mN \mM \mN^\top 
                \mSigmatrue^{-1/2} \mUtrue \right]_{k,\ell} \right|
\le
\frac{C \log T}{\SNR} 
+ \frac{ C }{T} \sqrt{ \frac{\log T}{\SNR} }
	\sqrt{ \sum_{i=1}^n \frac{ T \nu_i \log T }{ \sigmatrue{i}^2 } }
\le
\frac{C \log T}{\SNR} .
\end{equation*}
Squaring and summing over all $k,\ell \in [d]$,
\begin{equation} \label{eq:USNMNSU:done} 
\left\| \frac{1}{T} \mUtrue^\top \mSigmatrue^{-1/2}
                \mN \mM \mN^\top 
                \mSigmatrue^{-1/2} \mUtrue \right\|_F
\le \frac{C d \log T}{ \SNR} .
\end{equation}

Recalling $\mZobsd = \mZtrue + \mN$ and applying the triangle inequality
and Equations~\eqref{eq:UEU:NZbound} and~\eqref{eq:USNMNSU:done},
\begin{equation*}
\frac{1}{T}
\left\| \mUtrue^\top \mSigmatrue^{-1/2}
    \left( \mZobsd \mM \mZobsd^\top - \mZtrue \mM \mZtrue^\top \right)
    \mSigmatrue^{-1/2} \mUtrue \right\|_F
\le
C d \left( \sqrt{ \frac{ n \log T }{ T \SNR } } + \frac{\log T}{\SNR} \right) .
\end{equation*}
Applying this and Equation~\eqref{eq:USigmaTermsU:frob} to Equation~\eqref{eq:EUE:frob:maintri},
\begin{equation*}
\left\| \mUtrue^\top \left( \mRobsd - \mRtrue \right) \mUtrue \right\|_F
\le 
C d \left\| \mRtrue \right\|_F
    \left( \frac{ \log T }{ \SNR } + \sqrt{ \frac{ \log T }{ T \SNR } } \right)
+
C d \left( \frac{\log T}{\SNR} + \sqrt{ \frac{ n \log T }{ T \SNR } } \right).
\end{equation*}
Collecting terms and applying Lemma~\ref{lem:Rtrue:facts}  completes the proof.
\end{proof}

\begin{lemma} \label{lem:RobsdRtrueULam:mxtti}
Suppose that Assumptions~\ref{assum:N} and~\ref{assum:Rspec} hold.
Then
\begin{equation*}
\left\| \left( \mRobsd - \mRtrue \right) \mUtrue \mLambdatrue^{-1/2} \right\|_{\tti} 
\le 
C \sqrt{ \frac{ d \log T }{ \SNR \lambdatrue{d} } }
\left[ 1
+ \left( \sqrt{\frac{ \log n }{ \SNR }} + \sqrt{ \frac{ 1 }{ T } } \right) \| \mRtrue \|_{\tti}
\right],
\end{equation*}
\end{lemma}
\begin{proof}
Recalling the definitions of $\mRobsd$ and $\mRtrue$ from Equations~\eqref{eq:def:Robsd} and~\eqref{eq:def:Rtrue}, respectively, and applying the triangle inequality,
\begin{equation} \label{eq:bigexpand:term1:SigmaSplit} \begin{aligned}
&\left\| \left( \mRobsd - \mRtrue \right) \mUtrue \mLambdatrue^{-1/2} \right\|_{\tti} \\
&~~~~~~\le
\frac{1}{T}\left\| \mSigmaobsd^{-1/2} \left( \mZobsd \mM \mZobsd^\top 
			- \mZtrue \mM \mZtrue^\top \right)
		\mSigmaobsd^{-1/2} \mUtrue \mLambdatrue^{-1/2} \right\|_{\tti} \\
&~~~~~~~~~~~~+
\frac{1}{T}\left\| \left( \mSigmaobsd^{-1/2} \mZtrue \mM \mZtrue^\top \mSigmaobsd^{-1/2} 
	- \mSigmatrue^{-1/2} \mZtrue \mM \mZtrue^\top \mSigmatrue^{-1/2} 
	\right) \mUtrue \mLambdatrue^{-1/2} \right\|_{\tti} .
\end{aligned} \end{equation}
Recalling from Equation~\eqref{eq:def:Zobsd} that $\mZobsd = \mZtrue + \mN$ applying the triangle inequality, 
and using basic properties of the $(\tti)$-norm with Lemma~\ref{lem:sigmaratio},
\begin{equation} \label{eq:EUS:ZNterms} \begin{aligned}
&\frac{1}{T} \left\| \mSigmaobsd^{-1/2} \left( \mZobsd \mM \mZobsd^\top 
			- \mZtrue \mM \mZtrue^\top \right)
			\mSigmaobsd^{-1/2} \mUtrue \mLambdatrue^{-1/2} 
	\right\|_{\tti} \\
&~~~\le
\frac{C}{T} 
\left\| \mSigmatrue^{-1/2} \mN \mM 
	\mZtrue^\top \mSigmaobsd^{-1/2} \mUtrue \mLambdatrue^{-1/2} \right\|_{\tti}  \\
&~~~~~~+ \frac{C}{T}
	\left\| \mSigmatrue^{-1/2} \mZtrue \mM \mN^\top \mSigmaobsd^{-1/2} 
		\mUtrue \mLambdatrue^{-1/2} \right\|_{\tti} \\
&~~~~~~+ \frac{C}{T}
	\left\| \mSigmatrue^{-1/2} \mN \mM \mN^\top \mSigmaobsd^{-1/2} \mUtrue \mLambdatrue^{-1/2} 
		\right\|_{\tti} .
\end{aligned} \end{equation}

For ease of notation, recall $\mZtruetilde$ from Equation~\eqref{eq:def:Ztruetilde}.
We have
\begin{equation} \label{eq:EUS:ZNterms:NMZ:step1}
\begin{aligned}
\frac{1}{T}\!
\left\| \mSigmatrue^{-1/2} \mN \mM 
	\mZtrue \mSigmaobsd^{-1/2} \mUtrue \mLambdatrue^{-1/2} \right\|_{\tti}
\!=
\frac{1}{\sqrt{T}}\!
\left\| \mSigmatrue^{-1/2} \mN \mZtruetildetop \mSigmatrue^{1/2} 
		\mSigmaobsd^{-1/2} \mUtrue \mLambdatrue^{-1/2} \right\|_{\tti} .
\end{aligned} \end{equation}
Adding and subtracting appropriate quantities and applying the triangle inequality,
\begin{equation} \label{eq:EUS:ZNterms:NMZ:step2}
\begin{aligned}
&\left\| \mSigmatrue^{-1/2} \mN \mZtruetildetop \mSigmaobsd^{-1/2} 
			\mSigmatrue^{1/2} \mUtrue \mLambdatrue^{-1/2} \right\|_{\tti} \\
&~~~\le
\left\| \mSigmatrue^{-1/2} \mN \mZtildestartop 
		\mUtrue \mLambdatrue^{-1/2} \right\|_{\tti}
+
\left\| \mSigmatrue^{-1/2} \mN \mZtildestartop 
		\left( \mI - \mSigmaobsd^{-1/2} \mSigmatrue^{1/2} \right)
		\mUtrue \mLambdatrue^{-1/2} \right\|_{\tti} .
\end{aligned} \end{equation}
Applying Lemma~\ref{lem:tti:mx:highprob} with $\mB = \mZtildestartop \mUtrue \mLambdatrue^{-1/2}$,
\begin{equation} \label{eq:EUS:ZNterms:NMZ:step2:term1} \begin{aligned}
\left\| \mSigmatrue^{-1/2} \mN \mZtildestartop \mUtrue \mLambdatrue^{-1/2} \right\|_{\tti}  
&\le C \left\| \mZtildestartop \mUtrue \mLambdatrue^{-1/2} \right\|_F
	\sqrt{ \frac{\log n }{ \SNR } } 
\le C \sqrt{ \frac{ d \log n }{ \SNR } } ,
\end{aligned} \end{equation}
where the second inequality follows from the fact that
\begin{equation*}
\mZtildestar \mZtildestartop = \mRtrue
= \mUtrue \mLambdatrue \mUtrue^\top .
\end{equation*}

By basic properties of the $(\tti)$-norm,
\begin{equation*}
\left\| \mSigmatrue^{-1/2} \mN \mZtildestartop 
		\left( \mI - \mSigmaobsd^{-1/2} \mSigmatrue^{1/2} \right)
		\mUtrue \mLambdatrue^{-1/2} \right\|_{\tti}  
\le 
\frac{1}{\sqrt{\lambdatrue{d} }} 
	\left\| \mSigmatrue^{1/2} \mN \mZtildestartop \right\|_{\tti}
	\left\| \mI - \mSigmaobsd^{-1/2} \mSigmatrue^{1/2} \right\| .
\end{equation*} 
Applying Lemma~\ref{lem:tti:mx:highprob} with $\mB = \mZtildestartop$,
\begin{equation*} \begin{aligned}
\left\| \mSigmatrue^{-1/2} \mN \mZtildestartop 
		\left( \mI - \mSigmaobsd^{-1/2} \mSigmatrue^{1/2} \right)
		\mUtrue \mLambdatrue^{-1/2} \right\|_{\tti}  
&\le 
\frac{C}{\sqrt{\SNR \lambdatrue{d}}} \| \mZtildestar \|_F 
	\left\| \mI - \mSigmaobsd^{-1/2} \mSigmatrue^{1/2} \right\| \\
&\le C \sqrt{ \frac{ d \lambdatrue{1} \log n }{ \lambdatrue{d} \SNR } } 
	\left\| \mI - \mSigmaobsd^{-1/2} \mSigmatrue^{1/2} \right\| ,
\end{aligned} \end{equation*}
where the second inequality follows from the fact that the $d$ non-zero singular values of $\mZtildestar$ are precisely the square roots of the eigenvalues of $\mRtrue$.
Applying this and Equation~\eqref{eq:EUS:ZNterms:NMZ:step2:term1} to Equation~\eqref{eq:EUS:ZNterms:NMZ:step2},
\begin{equation*}
\left\| \mSigmatrue^{-1/2} \mN \mZtildestartop \mSigmaobsd^{-1/2} \mSigmatrue^{1/2}
			\mUtrue \mLambdatrue^{-1/2} \right\|_{\tti}
\le C \sqrt{ \frac{ d \log n }{ \SNR } } 
	\left( 1 + \kappa \left\| \mI - \mSigmaobsd^{-1/2} \mSigmatrue^{1/2} \right\| \right) .
\end{equation*}
Applying Lemma~\ref{lem:sigmaratio} and using our assumption in Equation~\eqref{eq:assum:noisepower} and our assumption that $\kappa = o(T)$,
\begin{equation*}
\left\| \mSigmatrue^{-1/2} \mN \mZtildestartop 
			\mSigmaobsd^{-1/2} \mSigmatrue^{1/2}
			\mUtrue \mLambdatrue^{-1/2} \right\|_{\tti}
\le C \sqrt{ \frac{ d \log n }{ \SNR } } .
\end{equation*}
Applying this to Equation~\eqref{eq:EUS:ZNterms:NMZ:step1},
\begin{equation} \label{eq:EUS:ZNterms:NMZ:done}
\frac{1}{T}
\left\| \mSigmatrue^{-1/2} \mN \mM \mZtrue \mSigmaobsd^{-1/2} \mUtrue \mLambdatrue^{-1/2} \right\|_{\tti}
\le C \sqrt{ \frac{ d \log n }{ T \SNR } }
\le C \sqrt{ \frac{ d \log n }{ \SNR \lambdatrue{d} } },
\end{equation}
where we have used the trivial bound $\lambdatrue{d} \le \lambdatrue{1}$ followed by Lemma~\ref{lem:Rtrue:facts} and Equation~\eqref{eq:assum:nT}.

By basic properties of the $(\tti)$-norm,
\begin{equation} \label{eq:EUS:ZNterms:ZMN:step1}
\frac{1}{T}
\left\| \mSigmaobsd^{-1/2} \mZtrue \mM \mN^\top \mSigmaobsd^{-1/2} \mUtrue \mLambdatrue^{-1/2} \right\|_{\tti}
\le
\frac{C}{\sqrt{T \lambdatrue{d}}}
\left\| \mZtildestar \mN^\top \mSigmaobsd^{-1/2} \mUtrue \right\|_{\tti}.
\end{equation}
Applying the triangle inequality and using basic properties of the $(\tti)$-norm,
\begin{equation} \label{eq:EUS:ZNterms:ZMN:triangle}
\left\| \mZtildestar \mN^\top \mSigmaobsd^{-1/2} \mUtrue \right\|_{\tti}
\le
\left\| \mZtildestar \mN^\top \mSigmatrue^{-1/2} \mUtrue \right\|_{\tti}
+ 
\left\|  \mZtildestar \mN^\top \mSigmatrue^{-1/2} \right\|_{\tti}
\left\| \mSigmatrue^{1/2} \mSigmaobsd^{-1/2} - \mI \right\| .
\end{equation}

Fixing $i \in [n]$ and $k \in [d]$,
\begin{equation*}
\left[ \mZtildestar \mN^\top \mSigmatrue^{-1/2} \mUtrue \right]_{ik}
= \sum_{j=1}^n \frac{ \langle \mZtildestar_{i\cdot}, \mN_j \rangle \Utrue{j,k} }{ \sigmatrue{j} } .
\end{equation*}
Since the rows of $\mZtildestar$ are all unit-norm and the rows of $\mN$ are subgaussian random vectors, standard concentration inequalities yield that with high probability, it holds for all $i \in [n]$ that
\begin{equation*}
\left| \left[ \mZtildestar \mN^\top \mSigmatrue^{-1/2} \mUtrue \right]_{i,k} \right|
\le C \sqrt{ \sum_{j=1}^n \frac{ \nu_j {\Utrue{j,k}}^2 \log n }
			{ \sigmatrue{j}^2 } }
\le C \sqrt{ \frac{ \log n }{ \SNR } } .
\end{equation*}
It follows that with high probability,
\begin{equation} \label{eq:EUS:ZNterms:ZMN:term1}
\left\| \mZtildestar \mN^\top \mSigmatrue^{-1/2} \mUtrue \right\|_{\tti}
= \Op{ \sqrt{ \frac{ d \log n }{ \SNR } } }.
\end{equation}
Similarly, 
\begin{equation*}
\left\|  \mZtildestar \mN^\top \mSigmatrue^{-1/2} \right\|_{\tti}
\le C \sqrt{ \sum_{j=1}^n \frac{ \nu_j \log n}{ \sigmatrue{j}^2 } }
\le C \sqrt{ \frac{ n \log n }{ \SNR } } ,
\end{equation*}
so that
\begin{equation*}
\left\|  \mZtildestar \mN^\top \mSigmatrue^{-1/2} \right\|_{\tti}
\left\| \mSigmatrue^{1/2} \mSigmaobsd^{-1/2} - \mI \right\|
\le C \sqrt{ \frac{ n \log n }{ \SNR } }
	\left\| \mSigmatrue^{1/2} \mSigmaobsd^{-1/2} - \mI \right\|.
\end{equation*}
Applying this and Equation~\eqref{eq:EUS:ZNterms:ZMN:term1} to Equation~\eqref{eq:EUS:ZNterms:ZMN:triangle},
\begin{equation*}
\left\| \mZtildestar \mN^\top \mSigmaobsd^{-1/2} \mUtrue \right\|_{\tti}
\le
C \sqrt{ \frac{ \log n }{ \SNR } }
\left( \sqrt{d}  
	+ \sqrt{n} \left\| \mSigmatrue^{1/2} \mSigmaobsd^{-1/2} - \mI \right\| 
\right) .
\end{equation*}
Applying this in turn to Equation~\eqref{eq:EUS:ZNterms:ZMN:step1},
\begin{equation} \label{eq:EUS:ZNterms:ZMN:done}
\begin{aligned}
\frac{1}{T}
\left\| \mSigmaobsd^{-1/2} \mZtrue \mM \mN^\top \mSigmaobsd^{-1/2} \mUtrue \mLambdatrue^{-1/2} \right\|_{\tti}
&\le
C \sqrt{ \frac{ \log n }{ \SNR \lambdatrue{d} } }
\left( \sqrt{ \frac{d}{T} }
	+ \left\| \mSigmatrue^{1/2} \mSigmaobsd^{-1/2} - \mI \right\| 
\right) \\
\le C\sqrt{ \frac{ \log n }{ \SNR \lambdatrue{d} } } ,
\end{aligned} \end{equation}
where the second inequality follows from Lemma~\ref{lem:sigmaratio}, Equations~\eqref{eq:assum:noisepower} and~\eqref{eq:assum:nT}, and our assumption that $\kappa = o(T)$.

By basic properties of the $(\tti)$-norm and Lemma~\ref{lem:sigmaratio},
\begin{equation} \label{eq:EUS:ZNterms:NMN:start}
\frac{1}{T} 
\left\| \mSigmaobsd^{-1/2} \mN \mM \mN^\top \mSigmaobsd^{-1/2} 
		\mUtrue \mLambdatrue^{-1/2} \right\|_{\tti}
\le
\frac{C}{T \sqrt{\lambdatrue{d}}}
\left\| \mSigmatrue^{-1/2} \mN \mM \mN^\top \mSigmaobsd^{-1/2} \mUtrue \right\|_{\tti}.
\end{equation}
Applying basic properties of the $(\tti)$-norm and using the triangle inequality,
\begin{equation} \label{eq:EUS:ZNterms:NMN:triangle} \begin{aligned}
\left\| \mSigmatrue^{-1/2} \mN \mM \mN^\top \mSigmaobsd^{-1/2} 
	\mUtrue \right\|_{\tti}
\!&\le\!
\left\| \mSigmatrue^{-1/2} \mN \mM\mN^\top \! \mSigmatrue^{-1/2} \!
	\left( \mSigmatrue^{1/2} \mSigmaobsd^{-1/2} \!-\! \mI \right) \! \mUtrue \right\|_{\tti} \\
&~~~~~~+ \!
\left\| \mSigmatrue^{-1/2} \mN \mM \mN^\top \mSigmatrue^{-1/2} \mUtrue \right\|_{\tti} .
\end{aligned} \end{equation}

For $i \in [n]$ and $k \in [d]$,
\begin{equation} \label{eq:NMNU:entry}
\left[ \mSigmatrue^{-1/2} \mN \mM \mN^\top \mSigmatrue^{-1/2} 
	\mUtrue \right]_{i,k} 
= \sum_{j=1}^n \frac{ \mN_i^\top \mM \mN_j \Utrue{j,k} }
                	{ \sigmatrue{i} \sigmatrue{j} }
= \frac{ \mN_i^\top \mM \mN_i \Utrue{i,k} }{ \sigmatrue{i}^2 }
+ \sum_{j\neq i} \frac{ \mN_i^\top \mM \mN_j \Utrue{j,k} }
                	{ \sigmatrue{i} \sigmatrue{j} } .
\end{equation}
Applying Cauchy-Schwarz followed by Lemma~\ref{lem:N:max},
\begin{equation} \label{eq:NMNU:entry:ii}
\left| \frac{ \mN_i^\top \mM \mN_i \Utrue{i,k} }{ \sigmatrue{i}^2 } \right|
\le
\frac{ \| \mN_i \|^2 |\Utrue{i,k}| }{ \sigmatrue{i}^2 }
\le \frac{ C \nu_i |\Utrue{i,k}| T \log T }{ \sigmatrue{i}^2 }
\le \frac{ C T \log T }{ \SNR } .
\end{equation}
Similarly, noting that the sum on the right-hand side of Equation~\eqref{eq:NMNU:entry} is a sum of independent sub-Gaussian random variables,
\begin{equation*}
\left| \sum_{j\neq i} \frac{ \mN_i^\top \mM \mN_j \Utrue{j,k} }
                	{ \sigmatrue{i} \sigmatrue{j} } \right|
\le 
\frac{ C \| \mN_i \| }{ \sigmatrue{i} }
	\sqrt{ \sum_{j \neq i} \frac{ \nu_j {\Utrue{j,k}}^2 }{ \sigmatrue{j}^2 } } 
\le
\frac{ C \| \mN_i \| }{ \sqrt{ \sigmatrue{i}^2 \SNR } } .
\end{equation*}
Again using Lemma~\ref{lem:N:max},
\begin{equation*}
\left| \sum_{j\neq i} \frac{ \mN_i^\top \mM \mN_j \Utrue{j,k} }
                	{ \sigmatrue{i} \sigmatrue{j} } \right|
\le \frac{ C \sqrt{\nu_i T \log T}}{ \sqrt{ \sigmatrue{i}^2 \SNR} }
\le \frac{ C \sqrt{T} \log T }{ \SNR } .
\end{equation*}
Applying this and Equation~\eqref{eq:NMNU:entry:ii} to Equation~\eqref{eq:NMNU:entry}, it holds with high probability that for all $i \in [n]$ and $k \in [d]$,
\begin{equation*}
\left| \left[ \mSigmatrue^{-1/2} \mN \mM \mN^\top \mSigmatrue^{-1/2} 
	\mUtrue \right]_{i,k} \right|
\le \frac{ C T \log T }{ \SNR } ,
\end{equation*}
from which
\begin{equation*}
\left\| \mSigmatrue^{-1/2} \mN \mM \mN^\top \mSigmatrue^{-1/2} 
        \mUtrue \right\|_{\tti}^2
\le C \max_{i \in [n]} \sum_{k=1}^d 
	\left[ \mSigmatrue^{-1/2} \mN \mM \mN^\top \mSigmatrue^{-1/2} 
        	\mUtrue \right]_{i,k}^2 
\le \frac{ C d T^2 \log^2 T }{ \SNR^2 } .
\end{equation*}
Taking square roots,
\begin{equation} \label{eq:NMN:tti:tri:term1:done}
\left\| \mSigmatrue^{-1/2} \mN \mM \mN^\top \mSigmatrue^{-1/2} 
        \mUtrue \right\|_{\tti}
\le \frac{ C \sqrt{d} T \log T }{ \SNR } .
\end{equation} 

Applying basic properties of the $(\tti)$-norm,
\begin{equation*} 
\left\| \mSigmatrue^{-1/2} \mN^\top
	\left( \mSigmatrue^{-1/2} \mSigmaobsd^{1/2} - \mI \right)
	\mUtrue \right\|_{\tti}
\le
\left\| \mSigmatrue^{-1/2} \mN^\top \right\|_{\tti}
\left\| \mSigmatrue^{-1/2} \mSigmaobsd^{1/2} - \mI \right\| .
\end{equation*}
Applying Lemma~\ref{lem:N:max},
\begin{equation*}
\left\| \mSigmatrue^{-1/2} \mN^\top
	\left( \mSigmatrue^{-1/2} \mSigmaobsd^{1/2} - \mI \right)
	\mUtrue \right\|_{\tti}
\le
C\sqrt{ \frac{ T \log T }{ \SNR } }
\left\| \mSigmatrue^{-1/2} \mSigmaobsd^{1/2} - \mI \right\| .
\end{equation*}
Applying this and Equation~\eqref{eq:NMN:tti:tri:term1:done} to Equation~\eqref{eq:EUS:ZNterms:NMN:triangle},
\begin{equation*}
\left\| \mSigmatrue^{-1/2} \mN \mM \mN^\top \mSigmaobsd^{-1/2} 
	\mUtrue \mLambdatrue^{-1/2} \right\|_{\tti}
\le
\sqrt{ \frac{ d T \log T }{ \SNR } }
\left(
1
+ 
\left\| \mSigmatrue^{-1/2} \mSigmaobsd^{1/2} - \mI \right\| \right).
\end{equation*}
Applying this in turn to Equation~\eqref{eq:EUS:ZNterms:NMN:start},
\begin{equation*} 
\frac{1}{T} 
\left\| \mSigmaobsd^{-1/2} \mN \mM \mN^\top \mSigmaobsd^{-1/2} 
		\mUtrue \mLambdatrue^{-1/2} \right\|_{\tti}
\le
C\sqrt{ \frac{ d\log T }{ \SNR \lambdatrue{d} } }
\left( 1 + 
	\left\| \mSigmatrue^{-1/2} \mSigmaobsd^{1/2} - \mI \right\| 
\right)
\end{equation*}
and Lemma~\ref{lem:sigmaratio}, along with Equations~\eqref{eq:assum:noisepower} and~\eqref{eq:assum:nT} and our assumption that $\kappa=o(T)$, yields
\begin{equation} \label{eq:EUS:ZNterms:NMN:done}
\frac{1}{T} 
\left\| \mSigmaobsd^{-1/2} \mN \mM \mN^\top \mSigmaobsd^{-1/2} 
		\mUtrue \mLambdatrue^{-1/2} \right\|_{\tti}
\le C\sqrt{ \frac{ d\log T }{ \SNR \lambdatrue{d} } } .
\end{equation}

Applying Equations~\eqref{eq:EUS:ZNterms:NMZ:done},
~\eqref{eq:EUS:ZNterms:ZMN:done}
and~\eqref{eq:EUS:ZNterms:NMN:done}
to Equation~\eqref{eq:EUS:ZNterms},
\begin{equation} \label{eq:EUS:ZNterms:done} \begin{aligned}
\frac{1}{T} \left\| \mSigmaobsd^{-1/2} \left( \mZobsd \mM \mZobsd^\top 
			- \mZtrue \mM \mZtrue^\top \right)
			\mSigmaobsd^{-1/2} \mUtrue \mLambdatrue^{-1/2} 
	\right\|_{\tti} 
&\le
C \sqrt{ \frac{ d \log T }{ \SNR \lambdatrue{d} } } .
\end{aligned} \end{equation}
Applying this in turn to Equation~\eqref{eq:bigexpand:term1:SigmaSplit},
\begin{equation} \label{eq:bigexpand:halfwayThere} \begin{aligned}
&\left\| \left( \mRobsd - \mRtrue \right) \mUtrue \mLambdatrue^{-1/2} \right\|_{\tti} \\
&~~\le
\frac{1}{T}\left\| \left( \! \mSigmaobsd^{-1/2} \mZtrue \mM \mZtrue^\top \!\mSigmaobsd^{-1/2} 
	- \mSigmatrue^{-1/2} \mZtrue \mM \mZtrue^\top \!\mSigmatrue^{-1/2} \! \right) 
	\! \mUtrue \mLambdatrue^{-1/2} \right\|_{\tti} 
+ C \sqrt{ \frac{ d \log T }{ \SNR \lambdatrue{d} } } .
\end{aligned} \end{equation}

For $i \in [n]$ and $k \in [d]$,
\begin{equation*} \begin{aligned}
\frac{1}{T} & \left[ 
		\left( \mSigmaobsd^{-1/2} \mZtrue \mM \mZtrue^\top \mSigmaobsd^{-1/2} 
		- \mSigmatrue^{-1/2} \mZtrue \mM \mZtrue^\top \mSigmatrue^{-1/2} 
		\right) \mUtrue \mLambdatrue^{-1/2} \right]_{i,k} \\
&= \sum_{j=1}^n \left( \frac{1}{\sigmaobsd{i} \sigmaobsd{j}} 
					- \frac{1}{\sigmatrue{i} \sigmatrue{j}} \right)
	\frac{ \vZtrue{i}^\top \mM \vZtrue{j} \Utrue{j,k} }
			{ T \sqrt{\lambdatrue{k}} } 
= \sum_{j=1}^n
	\left( \frac{\sigmatrue{i} \sigmatrue{j}}{\sigmaobsd{i} \sigmaobsd{j}} 
		- 1 \right) \frac{ \Rtrue_{i,j} \Utrue{j,k} }{ \sqrt{\lambdatrue{k}} } .
\end{aligned} \end{equation*}
Applying the Cauchy-Schwarz inequality,
\begin{equation*} \begin{aligned}
&\left[ \frac{1}{T} 
		\left( \mSigmaobsd^{-1/2} \mZtrue \mM \mZtrue^\top \mSigmaobsd^{-1/2} 
		- \mSigmatrue^{-1/2} \mZtrue \mM \mZtrue^\top \mSigmatrue^{-1/2} 
		\right) \mUtrue \mLambdatrue^{-1/2} \right]_{i,k}^2 \\
&\le \left[ \sum_{j=1}^n
	\left( \frac{\sigmatrue{i} \sigmatrue{j}}{\sigmaobsd{i} \sigmaobsd{j}} 
		- 1 \right)^2 \frac{ {\Utrue{j,k}}^2 }{ \lambdatrue{k} } \right]
	\left[ \sum_{j=1}^n {\Rtrue_{i,j}}^2 \right] 
\le \frac{ C }{\lambdatrue{k}}
		\left( \frac{ \log n }{ \SNR } 
				+ \sqrt{ \frac{ \log n }{ T \SNR } } \right)^2
		\sum_{j=1}^n {\Rtrue_{i,j}}^2 ,
\end{aligned} \end{equation*}
where the second inequality follows from Lemma~\ref{lem:sigmaprod:ratio}.

Summing over $k \in [d]$, taking a maximum over $i \in [n]$, and taking square roots,
\begin{equation*} \begin{aligned}
&\left\| \frac{1}{T} \left( \mSigmaobsd^{-1/2} \mZtrue \mM \mZtrue^\top \mSigmaobsd^{-1/2} 
        - \mSigmatrue^{-1/2} \mZtrue \mM \mZtrue^\top \mSigmatrue^{-1/2} 
        \right) \mUtrue \right\|_{\tti} \\
&~~~~~~\le C \sqrt{ \frac{ d }{ \lambdatrue{d} } }
	\left( \frac{ \log n }{ \SNR } + \sqrt{ \frac{ \log n }{ T \SNR } } \right) \| \mRtrue \|_{\tti} .
\end{aligned} \end{equation*}
Using Equation~\eqref{eq:assum:nT} and applying this to Equation~\eqref{eq:bigexpand:halfwayThere},
\begin{equation*}
\left\| \left( \mRobsd - \mRtrue \right) \mUtrue \mLambdatrue^{-1/2} \right\|_{\tti} 
\le 
\sqrt{ \frac{ d \log T }{ \SNR \lambdatrue{d} } }
\left[ 1
+ \left( \sqrt{\frac{ \log n }{ \SNR }} + \sqrt{ \frac{ 1 }{ T } } \right) \| \mRtrue \|_{\tti}
\right],
\end{equation*}
completing the proof.
\end{proof}

\section{Controlling eigenspaces} \label{apx:subspaces}

Here we collect results related to alignment of the leading eigenspace of $\mRobsd$ with the signal eigenspace of $\mRtrue$, to be used in the proof of Theorem~\ref{thm:XobsdXtrue:tti} in Appendix~\ref{apx:thm:tti}.
We draw from standard results in subspace geometry \citep[see, e.g.,][]{DavKah1970,Bhatia1997,CapTanPri2019,CaiZha2018}.
As elsewhere in this work, the main technical challenge (and hence the main departure from previous work) lies in the dependence structure in the error $\mRobsd - \mRtrue$.

\begin{lemma} \label{lem:sinTheta} 
Suppose that Assumption~\ref{assum:N} holds.
Then
\begin{equation} \label{eq:sinTheta:bound}
\min_{\mW \in \bbO_d} \left\| \mUobsd \mW - \mUtrue \right\|
\le \sqrt{2} \left\| \sin \mTheta \right\| 
\le
C \left( \frac{ \kappa \log n }{ \SNR }
        + \sqrt{ \frac{ \kappa \log n }{ \SNR \lambdatrue{d} } } \right) .
\end{equation}
\end{lemma}
\begin{proof}
Basic results in subspace geometry \citep{CaiZha2018,CapTanPri2019} yield the first inequality in Equation~\eqref{eq:sinTheta:bound}.
To establish the second, the Davis-Kahan theorem \citep[][Theorem VII.3.2]{Bhatia1997} implies
\begin{equation*} 
\left\| \sin \mTheta \right\|
\le \frac{ C \left\| \mRobsd - \mRtrue \right\| }{ \lambdatrue{d} } .
\end{equation*}
Applying Lemma~\ref{lem:RobsdRtrue:spectral} completes the proof.
\end{proof}

In order to align the estimated subspace $\mUobsd$ to the signal subspace $\mUtrue$, following \cite{SpectralDS}, define 
\begin{equation} \label{eq:def:H}
\mH = \mUobsd^\top \mUtrue \in \bbO_d
\end{equation}
and the related Procrustes alignment matrix
\begin{equation} \label{eq:def:Q} 
\mQ 
= \sgn \left( \mH \right)
= \sgn \left( \mUobsd^\top \mUtrue \right) ,
\end{equation}
where, for $\mH \in \R^{d \times d}$ with SVD $\mH = \mV_1 \mD \mV_2^\top$, $\sgn( \mB ) = \mV_1 \mV_2^\top$. 
The following lemma largely follows Lemma 4.15 in \cite{SpectralDS}.
The main change is that we must account for the different error structure in $\mRobsd-\mRtrue$ compared to the independent edges considered in that work.

\begin{lemma} \label{lem:SpectralDS:lem4.15:H}
Suppose that Assumptions~\ref{assum:N} and~\ref{assum:Rspec} hold.
Let $\mH$ and $\mQ$ be as in Equations~\eqref{eq:def:H} and~\eqref{eq:def:Q}, respectively.
Then
\begin{equation} \label{eq:Hbound:1}
\left\| \mH - \mQ \right\|
\le
\frac{ C \kappa \log n }{ \SNR }
\left[ \frac{ \kappa \log n }{ \SNR } + \frac{ 1 }{ \lambdatrue{d} } \right]
\end{equation}
and
\begin{equation} \label{eq:Hbound:2}
\left\| \mH^{-1} \right\| 
= \Op{ 1 } .
\end{equation}
\end{lemma}
\begin{proof}
Recall that by definition, $\mH = \mUobsd^\top \mUtrue$.
Let $\cos \mTheta \in \R^{d \times d}$ be the diagonal matrix of the cosines of the principal angles between $\mUobsd$ and $\mUtrue$ so that the SVD of $\mH$ is given by $\mH = \mV_1 \cos \mTheta ~\mV_2^\top$ \citep{Bhatia1997}.
By definition,
\begin{equation*}
\mQ = \sgn( \mH ) = \mV_1 \mV_2^\top ,
\end{equation*}
and it follows that
\begin{equation*}
\left\| \mH - \mQ \right\|
\le \left\| \cos \mTheta - \mI \right\|
\le \left\| \cos^2 \mTheta - \mI \right\|
= \left\| \sin^2 \mTheta \right\| .
\end{equation*}
Applying Lemma~\ref{lem:sinTheta} yields Equation~\eqref{eq:Hbound:1}.

By Weyl's inequality,
\begin{equation*}
\sigma_d( \mH )
\ge \sigma_d\left( \mQ \right) - \left\| \mH - \mQ \right\|
= 1 - \left\| \mH - \mQ \right\|,
\end{equation*}
where we have used the fact that $\mQ = \mV_1 \mV_2^\top$ has $d$ singular values all equal to $1$.
Applying Equation~\eqref{eq:Hbound:1} along with our growth assumptions in Equations~\eqref{eq:assum:noisepower} and~\eqref{eq:assum:lambdad:LB} yields Equation~\eqref{eq:Hbound:2}.
\end{proof}

The following two lemmas serve as analogues of the bounds in Lemma 17 of \cite{LyzTanAthParPri2017}, adapted to the correlation network setting.

\begin{lemma} \label{lem:Qinterchange} 
Suppose that Assumptions~\ref{assum:N} and~\ref{assum:Rspec} hold.
Then, with $\mQ \in \bbO_d$ as in Equation~\eqref{eq:def:Q}.
\begin{equation*}
\left\| \mLambdaobsd \mQ  - \mQ \mLambdatrue \right\|_F
\le
C \frac{ \sqrt{d} \kappa^2  \log n }{ \SNR }
\left( \frac{ \lambdatrue{1} \log n }{ \SNR } + 1 \right)
+ \left\| \mRtrue \right\|_F
\left( \frac{ \log T }{ \SNR } + \sqrt{ \frac{ \log T }{ T \SNR } } \right) .
\end{equation*}
\end{lemma}
\begin{proof}
Adding and subtracting appropriate quantities
\begin{equation*} \begin{aligned}
\mLambdaobsd \mQ 
&= \mLambdaobsd \left( \mQ - \mH \right) + \mLambdaobsd \mH 
= \mLambdaobsd \left( \mQ - \mH \right) + \mUobsd^\top \mRobsd \mUtrue,
\end{aligned} \end{equation*}
where the second equality follows from the definition of $\mH$.
Adding and subtracting appropriate quantities again,
\begin{equation} \label{eq:Qswap:pause} \begin{aligned}
\mLambdaobsd \mQ 
&= \mLambdaobsd \left( \mQ - \mH \right) 
	+ \mUobsd^\top \left( \mRobsd - \mRtrue \right) \mUtrue
	+ \mUobsd^\top \mRtrue \mUtrue \\
&= \mLambdaobsd \left( \mQ - \mH \right) 
	+ \left( \mUobsd - \mUtrue\mUtrue^\top \mUobsd \right)^\top 
		\left( \mRobsd - \mRtrue \right) \mUtrue \\
	&~~~+ \mUobsd^\top \mUtrue\mUtrue^\top 
		\left( \mRobsd - \mRtrue \right) \mUtrue
	+ \mUobsd^\top \mUtrue \mLambdatrue ,
\end{aligned} \end{equation}
where we have used the fact that $\mRtrue \mUtrue = \mUtrue \mLambdatrue$.
Recalling $\mH = \mUobsd^\top \mUtrue$ and writing
\begin{equation*}
\mUobsd^\top \mUtrue \mLambdatrue
= \left( \mH - \mQ \right) \mLambdatrue + \mQ \mLambdatrue,
\end{equation*}
rearranging Equation~\eqref{eq:Qswap:pause} yields
\begin{equation*} \begin{aligned}
\mLambdaobsd \mQ - \mQ \mLambdatrue
&= \mLambdaobsd \left( \mQ - \mH \right) 
	+ \left( \mUobsd - \mUtrue\mUtrue^\top \mUobsd \right)^\top 
		\left( \mRobsd - \mRtrue \right) \mUtrue \\
	&~~~+ \mUobsd^\top \mUtrue\mUtrue^\top 
		\left( \mRobsd - \mRtrue \right) \mUtrue
	+ \left( \mH - \mQ \right) \mLambdatrue .
\end{aligned} \end{equation*}
Applying the triangle inequality and using basic properties of the norm,
\begin{equation*} \begin{aligned}
\left\| \mLambdaobsd \mQ - \mQ \mLambdatrue \right\|_F
&\le \left\| \mLambdaobsd \right\|_F \left\| \mQ - \mH \right\|
	+ \left\| \left( \mUobsd - \mUtrue\mUtrue^\top \mUobsd \right)^\top 
		\left( \mRobsd - \mRtrue \right) \mUtrue \right\|_F \\
	&~~~+ \left\| \mH \right\|
		\left\| \mUtrue^\top \left( \mRobsd - \mRtrue \right) 
			\mUtrue \right\|_F
	+ \left\| \mH - \mQ \right\| \left\| \mLambdatrue \right\|_F .
\end{aligned} \end{equation*} 
Applying Lemma~\ref{lem:RobsdRtrue:spectral} and the growth bounds in Assumption~\ref{assum:Rspec},
\begin{equation*} \begin{aligned}
\left\| \mLambdaobsd \mQ - \mQ \mLambdatrue \right\|_F
&\le C \left\| \mLambdatrue \right\|_F \left\| \mQ - \mH \right\|
	+ \left\| \left( \mUobsd - \mUtrue\mUtrue^\top \mUobsd \right)^\top 
		\left( \mRobsd - \mRtrue \right) \mUtrue \right\|_F \\
&~~~+ \left\| \mH \right\|
	\left\| \mUtrue^\top \left( \mRobsd - \mRtrue \right) \mUtrue \right\|_F.
\end{aligned} \end{equation*}
Applying Lemmas~\ref{lem:UEU:frobenius} and~\ref{lem:SpectralDS:lem4.15:H},
\begin{equation*} \begin{aligned}
\left\| \mLambdaobsd \mQ - \mQ \mLambdatrue \right\|_F
&\le \left\| \left( \mUobsd - \mUtrue\mUtrue^\top \mUobsd \right)^\top 
		\left( \mRobsd - \mRtrue \right) \mUtrue \right\|_F \\
&~~~+ C \left\| \mLambdaobsd \right\|_F
\left[ \left( \frac{ \kappa \log n }{ \SNR } \right)^2 
	+ \frac{ \kappa \log n }{ \SNR \lambdatrue{d} } \right]
+ C \left\| \mRtrue \right\|_F
	\left( \frac{ \log T }{ \SNR } 
		+ \sqrt{ \frac{ \log T }{ T \SNR } } \right) .
\end{aligned} \end{equation*}
Again applying Lemma~\ref{lem:RobsdRtrue:spectral} and the growth bounds in Assumption~\ref{assum:Rspec},
\begin{equation} \label{eq:Qinter:checkpt} \begin{aligned}
\left\| \mLambdaobsd \mQ - \mQ \mLambdatrue \right\|_F
&\le \left\| \left( \mUobsd - \mUtrue\mUtrue^\top \mUobsd \right)^\top 
		\left( \mRobsd - \mRtrue \right) \mUtrue \right\|_F \\
	&~~~+ C \sqrt{d} \lambdatrue{1}
	\left[ \left( \frac{ \kappa \log n }{ \SNR } \right)^2 
        + \frac{ \kappa \log n }{ \SNR \lambdatrue{d} } \right]
	+ C \left\| \mRtrue \right\|_F
    \left( \frac{ \log T }{ \SNR } + \sqrt{ \frac{ \log T }{ T \SNR } } \right) .
\end{aligned} \end{equation}

Applying basic properties of the Frobenius norm followed by Lemmas~\ref{lem:RobsdRtrue:spectral}
\begin{equation*} \begin{aligned}
\left\| \left( \mUobsd - \mUtrue\mUtrue^\top \mUobsd \right)^\top 
        \left( \mRobsd - \mRtrue \right) \mUtrue \right\|_F
&\le \left\| \mUobsd - \mUtrue\mUtrue^\top \mUobsd \right\| \left\| \mRobsd - \mRtrue \right\|
	\left\| \mUtrue \right\|_F \\
&\le
C \sqrt{d} \left( \frac{ \lambdatrue{1} \log n }{ \SNR }
				+ \sqrt{ \frac{\lambdatrue{1} \log n}{\SNR} } \right)
	\left\| \mUobsd - \mUtrue\mUtrue^\top \mUobsd \right\| 
\end{aligned} \end{equation*}
Noting that the singular values of $\mUobsd - \mUtrue\mUtrue^\top \mUobsd$ are precisely the sines of the principal angles between $\mUobsd$ and $\mUtrue$, Lemma~\ref{lem:sinTheta} implies that
\begin{equation*} \begin{aligned}
\left\| \left( \mUobsd - \mUtrue\mUtrue^\top \mUobsd \right)^{\! \top} \!\!
        \left( \mRobsd - \mRtrue \right) \mUtrue \right\|_F
&\le
C \sqrt{d} \left( \frac{ \lambdatrue{1} \log n }{ \SNR }
				+ \sqrt{ \frac{\lambdatrue{1} \log n}{\SNR} } \right)
	 \left( \frac{ \kappa \log n }{ \SNR }
        + \sqrt{ \frac{ \kappa \log n }{ \SNR \lambdatrue{d} } } \right) .
\end{aligned} \end{equation*}

Applying this to Equation~\eqref{eq:Qinter:checkpt},
\begin{equation*} \begin{aligned}
\left\| \mLambdaobsd \mQ - \mQ \mLambdatrue \right\|_F
&\le 
C \sqrt{d} \left( \frac{ \lambdatrue{1} \log n }{ \SNR }
				+ \sqrt{ \frac{\lambdatrue{1} \log n}{\SNR} } \right)
	 \left( \frac{ \kappa \log n }{ \SNR }
        + \sqrt{ \frac{ \kappa \log n }{ \SNR \lambdatrue{d} } } \right) \\
&~~~+ C \sqrt{d} \lambdatrue{1}
	\left[ \left( \frac{ \kappa \log n }{ \SNR } \right)^2 
        + \frac{ \kappa \log n }{ \SNR \lambdatrue{d} } \right]
	+ C \left\| \mRtrue \right\|_F
    \left( \frac{ \log T }{ \SNR } + \sqrt{ \frac{ \log T }{ T \SNR } } \right) .
\end{aligned} \end{equation*}
Recalling that $\kappa = \lambdatrue{1}/\lambdatrue{d}$ and rearranging, using the fact that
\begin{equation*}
\left( \sqrt{ \frac{ \lambdatrue{1} \log n }{ \SNR } } + 1 \right)^2
\le C\left( 1 + \frac{ \lambdatrue{1} \log n }{ \SNR } \right),
\end{equation*}
we have
\begin{equation*} \begin{aligned}
\left\| \mLambdaobsd \mQ - \mQ \mLambdatrue \right\|_F
&\le 
\frac{ C \sqrt{d} \kappa^2  \log n }{ \SNR }
\left[ \frac{ \lambdatrue{1} \log n }
			{ \SNR } + 1 \right]
+ C \left\| \mRtrue \right\|_F
\left( \frac{ \log T }{ \SNR } + \sqrt{ \frac{ \log T }{ T \SNR } } \right) ,
\end{aligned} \end{equation*}
as we set out to show.
\end{proof} 

\begin{lemma} \label{lem:Qinter:sqrt} 
Suppose that Assumptions~\ref{assum:N} and~\ref{assum:Rspec} hold.
Then, with $\mQ$ as in Equation~\eqref{eq:def:Q},
\begin{equation*}
\left\| \mLambdaobsd^{-1/2} \mQ  - \mQ \mLambdatrue^{-1/2} \right\|_F
\le
C \kappa \sqrt{ \frac{ d \log T }{ \SNR \lambdatrue{d} } }
\left[
\kappa \sqrt{ \frac{ \log n }{ \SNR } }\!
\left( \frac{ \kappa \log n }{ \SNR } + \frac{1}{\lambdatrue{d}} \! \right)
+ \!
\left( \sqrt{ \frac{ \log T }{ \SNR } } + \! \sqrt{ \frac{ 1 }{ T } } \right)
\right] .
\end{equation*}
\end{lemma}
\begin{proof}
Expanding the Frobenius norm,
\begin{equation*} \begin{aligned}
\left\| \mLambdaobsd^{-1/2} \mQ  - \mQ \mLambdatrue^{-1/2} \right\|_F^2
&= \sum_{k=1}^d \sum_{\ell=1}^d
	\left( \frac{Q_{k,\ell}}{ \sqrt{\lambdaobsd{k}} }
		- \frac{Q_{k,\ell}}{ \sqrt{\lambdatrue{\ell}} } \right)^2
= \sum_{k=1}^d \sum_{\ell=1}^d
	Q_{k,\ell}^2
	\left( \frac{ \sqrt{\lambdatrue{\ell}} - \sqrt{\lambdaobsd{k}} }
		{ \sqrt{\lambdaobsd{k}} \sqrt{\lambdatrue{\ell}} }
	\right)^2 \\
&= \sum_{k=1}^d \sum_{\ell=1}^d 
	\frac{ 
	\left( Q_{k,\ell} \lambdatrue{\ell} -  Q_{k,\ell} \lambdaobsd{k} 
	\right)^2 }
	{ \lambdaobsd{k} \lambdatrue{\ell}
	\left( \sqrt{\lambdatrue{\ell}}+\sqrt{\lambdaobsd{k}} \right)^2 } .
\end{aligned} \end{equation*}
It follows that
\begin{equation*} 
\left\| \mLambdaobsd^{-1/2} \mQ  - \mQ \mLambdatrue^{-1/2} \right\|_F^2
\le \frac{ \left\| \mLambdaobsd \mQ  - \mQ \mLambdatrue \right\|_F^2 }
{ \left(\sqrt{\lambdatrue{d}}+\sqrt{\lambdaobsd{d}} \right)^2
	\lambdatrue{d} \lambdaobsd{d} } .
\end{equation*}
Applying Lemma~\ref{lem:RobsdRtrue:spectral} along with Equations~\eqref{eq:assum:noisepower} and~\eqref{eq:assum:lambdad:LB},
\begin{equation*}
\left\| \mLambdaobsd^{-1/2} \mQ  - \mQ \mLambdatrue^{-1/2} \right\|_F^2
\le
\frac{ C }{ \lambdatrue{d}^3 }
 \left\| \mLambdaobsd \mQ  - \mQ \mLambdatrue \right\|_F^2 .
\end{equation*}
Taking square roots and applying Lemma~\ref{lem:Qinterchange},
\begin{equation*}
\left\| \mLambdaobsd^{-1/2} \mQ  - \mQ \mLambdatrue^{-1/2} \right\|_F
\le
\frac{ C }{ \lambdatrue{d}^{3/2} } \!
\left[
\frac{ \sqrt{d} \kappa^2  \log n }{ \SNR } \!
\left( \frac{ \lambdatrue{1} \log n }{ \SNR } + 1 \! \right)
+ \left\| \mRtrue \right\|_F \!
\left( \frac{ \log T }{ \SNR } + \! \sqrt{ \frac{ \log T }{ T \SNR } } \right)
\right] .
\end{equation*}
Simplifying and using our assumption that $n = O(T)$,
\begin{equation*}
\left\| \mLambdaobsd^{-1/2} \mQ  - \mQ \mLambdatrue^{-1/2} \right\|_F
\le
C \kappa \sqrt{ \frac{ d \log T }{ \SNR \lambdatrue{d} } }
\left[
\kappa \sqrt{ \frac{ \log n }{ \SNR } }\!
\left( \frac{ \kappa \log n }{ \SNR } + \frac{1}{\lambdatrue{d}} \! \right)
+ \!
\left( \sqrt{ \frac{ \log T }{ \SNR } } + \! \sqrt{ \frac{ 1 }{ T } } \right)
\right] ,
\end{equation*}
as we set out to show.
\end{proof}

\section{Proof of Main Result} \label{apx:thm:tti}

Using the supporting results established in Appendices~\ref{apx:sigma} through~\ref{apx:subspaces}, we are ready to prove our main theorem.

\begin{proof}[Proof of Theorem~\ref{thm:XobsdXtrue:tti}]
From the definitions in Equations~\eqref{eq:def:Xtrue} and~\eqref{eq:def:mXobsd},
\begin{equation*}
\mXobsd =  \mRobsd \mUobsd \mLambdaobsd^{-1/2}
~\text{ and }~
\mXtrue = \mRtrue \mUtrue \mLambdatrue^{-1/2} .
\end{equation*}
Thus, recalling $\mQ \in \bbO_d$ from Equation~\eqref{eq:def:Q}, the triangle inequality implies
\begin{equation} \label{eq:XhatX:tti:start} \begin{aligned}
\left\| \mXobsd \mQ - \mXtrue \right\|_{\tti}
&\le
\left\| \left( \mRobsd - \mRtrue \right) \mUtrue \mLambdatrue^{-1/2} \right\|_{\tti} 
+
\left\| \mRobsd \mUobsd \mQ \left( \mQ^\top \mLambdaobsd^{-1/2} \mQ - \mLambdatrue^{-1/2} \right) \right\|_{\tti} \\
&~~~~~~+ \left\| \mRobsd \left( \mUobsd \mQ - \mUtrue \right) \mLambdatrue^{-1/2} \right\|_{\tti}
\end{aligned} \end{equation}

Applying Lemma~\ref{lem:RobsdRtrueULam:mxtti},
\begin{equation} \label{eq:XhatX:tti:term1:done}
\left\| \left( \mRobsd - \mRtrue \right) \mUtrue \mLambdatrue^{-1/2} \right\|_{\tti}
\le 
C \sqrt{ \frac{ d \log T }{ \SNR \lambdatrue{d} } }
\left[ 1
	+ \left( \sqrt{\frac{ \log n }{ \SNR }} 
		+ \sqrt{ \frac{ 1 }{ T } } \right) \| \mRtrue \|_{\tti}
\right] .
\end{equation}

Applying basic properties of the $(\tti)$-norm,
\begin{equation*} 
\left\| \mRobsd \mUobsd \mQ \left( \mQ^\top \mLambdaobsd^{-1/2} \mQ - \mLambdatrue^{-1/2} \right) \right\|_{\tti}
\le
\left\| \mRobsd \right\|_{\tti}
\left\| \mQ^\top \mLambdaobsd^{-1/2} \mQ - \mLambdatrue^{-1/2} \right\| .
\end{equation*}
Applying Lemmas~\ref{lem:RobsdRtrue:tti} and~\ref{lem:Qinter:sqrt},
\begin{equation*} \begin{aligned}
&\left\| \mRobsd \mUobsd \mQ \left( \mQ^\top \mLambdaobsd^{-1/2} \mQ 
		- \mLambdatrue^{-1/2} \right) \right\|_{\tti} \\
&~~~~~\le
C \kappa \left\| \mRtrue \right\|_{\tti}
\sqrt{ \frac{ d \log T }{ \SNR \lambdatrue{d} } }
\left[
\kappa \sqrt{ \frac{ \log n }{ \SNR } }\!
\left( \frac{ \kappa \log n }{ \SNR } + \frac{1}{\lambdatrue{d}} \! \right)
+ \!
\left( \sqrt{ \frac{ \log T }{ \SNR } } + \! \sqrt{ \frac{ 1 }{ T } } \right)
\right] .
\end{aligned} \end{equation*}
Simplifying and using the assumptions in Equations~\eqref{eq:assum:noisepower} and~\eqref{eq:assum:lambdad:LB},
\begin{equation} \label{eq:XhatX:tti:term2:done} \begin{aligned}
&\left\| \mRobsd \mUobsd \mQ \! 
	\left( \mQ^\top \! \mLambdaobsd^{-1/2} \mQ \!
		- \mLambdatrue^{-1/2} \!\right) \right\|_{\tti} 
\! \le
C \kappa
\sqrt{ \frac{ d \log T }{ \SNR \lambdatrue{d} } } \!
\left( \kappa \sqrt{ \frac{ \log n }{ \SNR } }
 	+ \sqrt{ \frac{ 1 }{ T } } \right)
\left\| \mRtrue \right\|_{\tti} \! .
\end{aligned} \end{equation}

Recalling the definition of $\mRobsd$ from Equation~\eqref{eq:def:Robsd},
\begin{equation*} \begin{aligned}
\left\| \mRobsd \left( \mUobsd \mQ - \mUtrue \right) \mLambdatrue^{-1/2} \right\|_{\tti}
&= \left\| \mUobsd \mLambdaobsd 
		\left( \mQ - \mUobsd^\top \mUtrue \right) \mLambdatrue^{-1/2} \right\|_{\tti} \\
&= \left\| \mUobsd \mLambdaobsd \mUobsd^\top \mUobsd
		\left( \mQ - \mUobsd^\top \mUtrue \right) \mLambdatrue^{-1/2} \right\|_{\tti} .
\end{aligned} \end{equation*}
It follows that, by basic properties of the $(\tti)$-norm,
\begin{equation*}
\left\| \mRobsd \left( \mUobsd \mQ - \mUtrue \right) \mLambdatrue^{-1/2} \right\|_{\tti}
\le \frac{ \left\| \mRobsd \mUobsd \right\|_{\tti} }{ \sqrt{ \lambdatrue{d} } }
		\left\| \mQ - \mUobsd^\top \mUtrue \right\| 
\le \frac{ \left\| \mRobsd \right\|_{\tti} }{ \sqrt{ \lambdatrue{d} } }
		\left\| \mQ - \mUobsd^\top \mUtrue \right\| .
\end{equation*}
Applying Lemmas~\ref{lem:RobsdRtrue:tti} and~\ref{lem:SpectralDS:lem4.15:H},
\begin{equation*} 
\left\| \mRobsd \left( \mUobsd \mQ - \mUtrue \right) \mLambdatrue^{-1/2} \right\|_{\tti}
\le 
\frac{ C \left\| \mRtrue \right\|_{\tti} }{ \sqrt{ \lambdatrue{d} } }
\frac{\kappa \log n }{ \SNR }
\left( \frac{ \kappa \log n }{ \SNR } + \frac{ 1 }{ \lambdatrue{d} } \right).
\end{equation*}
Simplifying, using the assumptions in Equations~\eqref{eq:assum:noisepower} and~\eqref{eq:assum:lambdad:LB} and the trivial $d \ge 1$,
\begin{equation} \label{eq:XhatX:tti:term3:done}
\left\| \mRobsd \left( \mUobsd \mQ - \mUtrue \right) \mLambdatrue^{-1/2} \right\|_{\tti}
\le 
C \kappa \sqrt{ \frac{ d \log n }{ \SNR \lambdatrue{d} } }
\left\| \mRtrue \right\|_{\tti} 
\sqrt{ \frac{\log n }{ \SNR } } .
\end{equation}

Applying Equations~\eqref{eq:XhatX:tti:term1:done},
~\eqref{eq:XhatX:tti:term2:done}
and~\eqref{eq:XhatX:tti:term3:done}
to Equation~\eqref{eq:XhatX:tti:start}, using our assumption that $n = O(T)$ and simplifying,
\begin{equation*} \begin{aligned}
\left\| \mXobsd \mQ - \mXtrue \right\|_{\tti}
&\le 
C \sqrt{ \frac{ d \log T }{ \SNR \lambdatrue{d} } }
\left[ 1
	+ \kappa \left( \kappa \sqrt{\frac{ \log n }{ \SNR }} 
		+ \sqrt{ \frac{ 1 }{ T } } \right) \| \mRtrue \|_{\tti} 
\right] ,
\end{aligned} \end{equation*}
as we set out to show.
\end{proof}

\end{document}